\documentclass[11pt]{article}
\usepackage{color}
\usepackage{mathrsfs}
\usepackage{amsmath,amsthm,amssymb,amscd}
\usepackage{booktabs} 
\usepackage{float}
\usepackage[hyperfootnotes=false, colorlinks, linkcolor={blue}, citecolor={magenta}, filecolor={blue}, urlcolor={blue}, plainpages=false, pdfpagelabels]{hyperref}
\usepackage[overload]{textcase} 
\usepackage{mathtools}
\usepackage[numbers,sort&compress]{natbib} 
\usepackage{etoolbox} 
\apptocmd{\sloppy}{\hbadness 10000\relax}{}{} 
\usepackage[left=2.5cm, right=2.5cm, top=3cm]{geometry}
\allowdisplaybreaks
\usepackage{bm}

\usepackage{colonequals}
\newcommand{\A}{{\mathbb A}}
\newcommand{\Q}{{\mathbb Q}}
\newcommand{\Z}{{\mathbb Z}}
\newcommand{\R}{{\mathbb R}}
\newcommand{\C}{{\mathbb C}}

\newcommand{\p}{\mathfrak p}
\newcommand{\OF}{{\mathfrak o}}
\newcommand{\GL}{{\rm GL}}
\newcommand{\PGL}{{\rm PGL}}
\newcommand{\SL}{{\rm SL}}
\newcommand{\SO}{{\rm SO}}
\newcommand{\Sp}{{\rm Sp}}
\newcommand{\GSp}{{\rm GSp}}
\newcommand{\PGSp}{{\rm PGSp}}

\newcommand{\St}{{\rm St}}
\newcommand{\triv}{{\mathbf1}}

\newcommand{\Kl}{\Gamma_0'}
\newcommand{\B}{B}

\newcommand{\forget}[1]{}
\def\qdots{\mathinner{\mkern1mu\raise0pt\vbox{\kern7pt\hbox{.}}\mkern2mu
\raise3.4pt\hbox{.}\mkern2mu\raise7pt\hbox{.}\mkern1mu}}

\newtheorem{lemma}{Lemma}[section]
\newtheorem{theorem}[lemma]{Theorem}
\newtheorem{corollary}[lemma]{Corollary}
\newtheorem{proposition}[lemma]{Proposition}
\newtheorem{definition}[lemma]{Definition}
\newtheorem{remark}[lemma]{Remark}

\newcommand\blfootnote[1]{
	\begingroup
	\renewcommand\thefootnote{}\footnote{#1}
	\addtocounter{footnote}{-1}
	\endgroup
}
\usepackage[toc,page, title, titletoc]{appendix}
\makeatother
\makeatletter
\newcommand\appendix@section[1]{
	\refstepcounter{section}
	\orig@section*{Appendix \@Alph\c@section: #1}
	\addcontentsline{toc}{section}{Appendix \@Alph\c@section: #1}
}
\g@addto@macro\appendix{\let\section\appendix@section}
\let\orig@section\section
\makeatother
\title{On counting cuspidal automorphic representations for $\GSp(4)$}
\author{Manami Roy, Ralf Schmidt, and Shaoyun Yi}
\date{}

\begin{document}
	
\maketitle
\blfootnote{\hspace*{0.2in } Key words and phrases. Plancherel measure, cuspidal automorphic representations, Siegel cusp forms, dimension formula, Arthur packets.}
\begin{abstract}
\noindent We find the number $s_k(p,\Omega)$ of cuspidal automorphic representations of $\GSp(4,\A_{\Q})$ with trivial central character such that the archimedean component is a holomorphic discrete series representation of weight $k\ge 3$, and the non-archimedean component at $p$ is an Iwahori-spherical representation of type $\Omega$ and unramified otherwise. Using the automorphic Plancherel density theorem, we show how a limit version of our formula for $s_k(p,\Omega)$ generalizes to the vector-valued case and a finite number of ramified places.
\end{abstract}

\tableofcontents

\section{Introduction}
\noindent
As is well known, classical modular forms are associated to automorphic representations of the adelic group $\GL(2,\A_{\Q})$. Similarly, Siegel modular forms of degree $2$ are related to automorphic representations of the adelic group $\GSp(4,\A_{\Q})$; see \cite{AsgariSchmidt2001}. The latter have trivial central character, and can hence be regarded as automorphic representations of $\PGSp(4,\A_{\Q})$. The split orthogonal group $\SO(5)$, which is isomorphic to $\PGSp(4)$ as algebraic groups, is one of the groups for which Arthur \cite{Arthur2013} has provided a classification of the discrete automorphic spectrum in terms of the automorphic representations of general linear groups. (This classification has been extended to $\GSp(4,\A_{\Q})$ by Gee and Ta\"ibi in \cite{GeeOlivier2019}.) In the following let $G=\GSp(4)$.
\begin{definition}\label{Definition of mOmegak}
 Let $k$ be a positive integer, and let $p$ be a prime.
	Let $S_k(p,\Omega)$ be the set of cuspidal automorphic representations $\pi\cong\otimes_{v \le \infty} \pi_v$ of $G(\A_{\Q})$ with trivial central character satisfying the following properties:
	\begin{enumerate}
		\item  $\pi_{\infty}$ is the lowest weight module with minimal $K$-type $(k,k)$; it is a holomorphic discrete series representation if $k\geq3$, a holomorphic limit of discrete series representation if $k=2$, and a non-tempered representation if $k=1$. (It was denoted by $\mathcal{B}_{k,0}$ in \cite[Sect.~3.5]{Schmidt2017}.)
		\item    $\pi_{v}$ is unramified for each $v\neq p, \infty$.
		\item    $\pi_{p}$ is an Iwahori-spherical representation of $G(\Q_p)$ of type $\Omega$.
	\end{enumerate}
\end{definition}
Here the representation type $\Omega$ is one of the types listed in Table~\ref{rep with Arthur type}: I, IIa, IIb, \ldots. This table lists all the irreducible, admissible representations of $G(F)$ supported in the minimal parabolic subgroup, where $F$ is a non-archimedean local field of characteristic zero. By general principles it is known that $S_k(p,\Omega)$ is finite. Let
\begin{align}
\label{no. of cup form}
s_k(p,\Omega)&\colonequals \# S_k(p,\Omega)
\end{align}
be its cardinality. In this paper we will find an explicit formula for $s_k(p,\Omega)$ except for $k=2$. Moreover, we give some partial results for $k=2$ in Theorem~\ref{Theorem of I and IIb}, Theorems~\ref{Theorem of Saito-Kurokawa type} - \ref{relations of Yoshida type} and Proposition~\ref{k2prop}.

In order to do so, we explore the relationship between Siegel cusp forms of degree $2$ and cuspidal automorphic representations of $G(\A_{\Q})$  in Sect.~\ref{Section on Iwahori-spherical representations and Arthur packets}. Using local representation theory of $G(\Q_p)$  (see \cite{RobertsSchmidt2007}), we get a system of equations involving the quantities $s_k(p,\Omega)$ and the global dimensions for the spaces of Siegel cusp forms $S_k(\Gamma_p)$ of degree $2$ under the various congruence subgroups $\Gamma_p$ defined in \eqref{classical congruence subgroups} further below. We partition the set $S_k(p,\Omega)$ according to the six Arthur types for $\GSp(4)$ and define the Arthur type versions of the quantities $s_k(p,\Omega)$ in Sect.~\ref{Section on Arthur packets}. We show that these quantities are all zero for the Arthur packets of types {\bf(B)} and {\bf(Q)} in Sect.~\ref{Section on B and Q types}.  

Many mathematicians have studied the dimension formulas for the spaces of Siegel modular forms of degree $2$; see for example the references in Table~\ref{historytable}.
We consider scalar-valued Siegel cusp forms $S_k(\Gamma_p)$ of degree $2$, weight $k$ and level~$p$.  Using dimension formulas for these spaces, we compute a general formula for $s_k(p,\Omega)$ and a rational expression for the generating series $\sum_{k\ge 3} s_k(p,\Omega)t^k$ in Sect.~\ref{countPYsec} and \ref{countGsec}. Hence we find the numbers $s_k(p,\Omega)$ in a uniform way, where $\Omega$ varies over the representation types described in Table~\ref{rep with Arthur type}. Wakatsuki studied these cuspidal automorphic representations for the square-integrable representation types $\Omega=\rm IVa, Va$ in \cite{Wakatsuki2013}.

Furthermore, using the global newform theory for various congruence subgroups of $G(\Q)$ defined in \cite[Sect.~3.3]{Schmidt200501}, we find  the dimensions of the spaces of newforms $S_k^{\rm new}(\Gamma_p)$ in terms of the quantities $s_k(p,\Omega)$ in Sect.~\ref{newformssec}. 

In the final section we study a connection between the quantities $s_k(p,\Omega)$ and the total Plancherel measure $m_{\Omega}$, defined in \eqref{Plancherel measure Iwahori-spherical}, of the tempered Iwahori-spherical representations of $\PGSp(4,\Q_p)$ of type $\Omega$. In Sect.~\ref{Relationship with the Plancherel measure} we compute the $m_\Omega$, working over any non-archimedean local field $F$ of characteristic zero. It turns out (see \eqref{relationship of sk and Plancherel measure} and \eqref{amequaleq}) that the leading terms in the \emph{global} formulas for $s_k(p,\Omega)$ are proportional to the purely \emph{local} quantities $m_\Omega$.
We will show how this is a consequence of the automorphic Plancherel density theorem from \cite{Shin2012}, specialized to $\PGSp(4)$, even though \cite{Shin2012} works on the level of real $L$-packets while we require a holomorphic discrete series representation at infinity. Finally, we will use the results of \cite{Shin2012} to generalize the limit version of our formula for $s_k(p,\Omega)$ to the vector-valued case and to more than one ramified place; see Theorem~\ref{Theorem gerenal PDT}.

\section*{Notation}\label{sec2: Definitions and notations}
We let
\begin{equation}\label{GSp4 classical version}
G=\GSp(4) \colonequals \{g\in\GL(4)\colon \:^tgJg=\lambda(g)J,\:\lambda(g)\in \GL(1)\}, \quad J=\begin{bsmallmatrix} &&&1\\&&1&\\&-1&&\\-1&&&\end{bsmallmatrix}.
\end{equation}
The function $\lambda$ is called the multiplier homomorphism. The kernel of this function is the symplectic group $\Sp(4)$. Let $Z$ be the center of $\GSp(4)$ and $\PGSp(4)=\GSp(4)/Z$.

\textbf{Congruence subgroups of $\GSp(4,F)$.}
Let $F$ be a non-archimedean local field of characteristic zero. Let $\OF$ be the ring of integers of $F$, and let $\p$ be the maximal ideal of $\OF$. We fix a generator $\varpi$ of $\p$. Let $q$ be the cardinality of $\OF/\p$, and let $\nu$ be the normalized absolute value on $F$; thus $\nu(\varpi)=q^{-1}$. We consider the following congruence subgroups of $\GSp(4, F)$: the \textit{paramodular group} ${\rm K}(\p)$ of level $\p$, the \textit{Siegel congruence subgroup} ${\rm Si}(\p)$ of level $\p$, the \textit{Klingen congruence subgroup} ${\rm Kl}(\p)$ of level $\p$ and the \textit{Iwahori subgroup} $I$, which are defined as follows.
\begin{equation}\label{congruence subgroups for GSp4F}
\begin{split}
 K&:=\GSp(4,\OF),\\
{\rm{K}}(\p)&\colonequals \{g\in\GSp(4,F)\colon \det(g)\in\OF^{\times}\}\cap \left[\begin{smallmatrix}
\OF&\OF&\OF&\p^{-1} \\ \p&\OF&\OF&\OF \\ \p&\OF&\OF&\OF \\ \p&\p&\p&\OF
\end{smallmatrix}\right],\\
{\rm Si}(\p) &\colonequals \GSp(4,\OF)\cap\left[\begin{smallmatrix}
\OF&\OF&\OF&\OF \\ \OF&\OF&\OF&\OF \\ \p&\p&\OF&\OF \\ \p&\p&\OF&\OF
\end{smallmatrix}\right],\\
{\rm Kl}(\p) &\colonequals \GSp(4,\OF)\cap\left[\begin{smallmatrix}
\OF&\OF&\OF&\OF \\ \p&\OF&\OF&\OF \\ \p&\OF&\OF&\OF \\ \p&\p&\p&\OF
\end{smallmatrix}\right],\\
I &\colonequals \GSp(4,\OF)\cap\left[\begin{smallmatrix}
\OF&\OF&\OF&\OF \\ \p&\OF&\OF&\OF \\ \p&\p&\OF&\OF \\ \p&\p&\p&\OF
\end{smallmatrix}\right].
\end{split}
\end{equation}

An admissible representation of $G(F)$ is called \textit{\textit{Iwahori-spherical}} if it has non-zero $I$-invariant vectors. These are exactly the constituents of the representations parabolically induced from an unramified character of the Borel subgroup. The complete list of all irreducible, admissible representations of $G(F)$ that are constituents of Borel-induced representations is given in Table~\ref{rep with Arthur type}. 

Let $(\pi, V)$ be an Iwahori-spherical representation of $G(F)$, and let $H$ be one of the subgroups defined in \eqref{congruence subgroups for GSp4F}. We denote by $V^H$, and sometimes by $\pi^H$, the space of $H$-fixed vectors. The quantity $\dim(V^H)$ is the same across all Iwahori-spherical representations of type $\Omega$; here $\Omega\in\{\mathrm{I},\mathrm{IIa},\ldots,\mathrm{VId}\}$ is one of the types in Table~\ref{rep with Arthur type}. We denote this common dimension by $d_{H,\Omega}$. These numbers are given explicitly in \cite[Table 3]{Schmidt200501}.

\textbf{Congruence subgroups for Siegel modular forms.} Let $S_{k}(\Gamma)$ be the space of Siegel cusp forms of degree $2$ and weight $k$ with respect to a congruence subgroup $\Gamma$ of $\Sp(4,\Q)$. (When speaking about Siegel modular forms, it is more convenient to realize symplectic groups using the symplectic form $J=\begin{bsmallmatrix} 0&1_2\\-1_2&0\end{bsmallmatrix}$). The following congruence subgroups, which correspond to the local groups in \eqref{congruence subgroups for GSp4F}, will be of particular interest:

\begin{equation}\label{classical congruence subgroups}
\begin{split}
\text{The full modular group }  \Sp(4,\Z),&\\
\text{The paramodular group of level } N:
{\rm K}(N)&=\begin{bsmallmatrix} \Z&N\Z&\Z&\Z\\ \Z&\Z&\Z&N^{-1}\Z\\ \Z&N\Z&\Z&\Z
\\ N\Z&N\Z&N\Z&\Z\end{bsmallmatrix}\ \cap\  \Sp(4,\Q),\\
\text{The Siegel congruence subgroup of level } N:
\Gamma_0(N)&=\begin{bsmallmatrix} \Z&\Z&\Z&\Z\\\Z&\Z&\Z&\Z\\ N\Z&N\Z&\Z&\Z
\\ N\Z&N\Z&\Z&\Z\end{bsmallmatrix} \ \cap\  \Sp(4,\Z),\\
\text{The Klingen congruence subgroup of level } N: 
\Kl(N)&=\begin{bsmallmatrix} \Z&N\Z&\Z&\Z\\\Z&\Z&\Z&\Z\\ \Z&N\Z&\Z&\Z
\\ N\Z&N\Z&N\Z&\Z\end{bsmallmatrix} \ \cap\  \Sp(4,\Z),\\
\text{The Borel congruence subgroup of level } N: 
\B(N)&=\begin{bsmallmatrix} \Z&N\Z&\Z&\Z\\\Z&\Z&\Z&\Z\\ N\Z&N\Z&\Z&\Z
\\ N\Z&N\Z&N\Z&\Z\end{bsmallmatrix} \ \cap\  \Sp(4,\Z).
\end{split}
\end{equation}

\section{Iwahori-spherical representations and Arthur packets}\label{Section on Iwahori-spherical representations and Arthur packets}
In this section, we first discuss the relationship between $s_k(p,\Omega)$ defined as in \eqref{no. of cup form} and the dimensions $\dim_{\C}S_k(\Gamma)$, where $\Gamma$ is one of the congruence subgroups in~\eqref{classical congruence subgroups} of prime level $p$. We introduce the refined quantities $s_k^{\rm (\mathbf{*})}(p,\Omega)$, where {\bf($*$)} is one of five types of Arthur packets.

\subsection{Siegel modular forms associated to the representations in \texorpdfstring{$S_k(p,\Omega)$}{}}
We review the connection between Siegel modular forms of degree $2$ and automorphic representations of $G(\A_{\Q})$; for more details see \cite{AsgariSchmidt2001} and \cite[Sect.~3.2]{Schmidt2017}. 

Let $\pi\cong\otimes_{p \le \infty}\pi_p$ be a cuspidal automorphic representation of $G(\A_{\Q})$ with trivial central character, where $\pi_p$ is an irreducible admissible representation of $G(\Q_p)$. Let $V_p$ be a model for $\pi_p$, so that $V\cong \otimes_{p} V_p$, a restricted tensor product. In order to get a holomorphic scalar-valued Siegel modular form of weight $k$, we need to make an assumption that $\pi_\infty$ is a holomorphic discrete series representation with minimal $K$-type $(k,k)$. Let $v_\infty \in V_\infty$ be a non-zero vector of weight $(k, k)$ in this $K$-type. For each finite prime $p$, let $v_p$ be a non-zero vector in $V_p$, and let $C_p$ be an open-compact subgroup of $G(\Q_p)$ stabilizing $v_p$. For almost all primes $p$, we assume that $v_p$ is the distinguished unramified vector and $C_p=G(\Z_p)$. By our choices, $\otimes_{p} v_p$ is a legitimate element in $\otimes_{p} V_p$ and it corresponds to a cusp form $\Phi \in V$ via the isomorphism $V\cong \otimes_{p} V_p$. Using strong approximation for $\Sp(4)$, the automorphic form $\Phi$ gives rise to a cuspidal Siegel eigenform $f$ of degree $2$ and weight $k$ with respect to the congruence subgroup $\Gamma=G(\Q) \cap G(\R)^{+}\prod_{p<\infty}C_p$ of $\Sp(4,\Q)$. Every eigenform in $S_k(\Gamma)$ arises in this way.

In particular, consider $\pi\cong\otimes\pi_v\in S_k(p,\Omega)$. Recall that $\pi_p$ is an Iwahori-spherical representation of type $\Omega$. Let $C_p$ be one of the compact open subgroups in \eqref{congruence subgroups for GSp4F}, and let $\Gamma_p$ be the corresponding congruence subgroup of $\Sp(4,\Q)$, as in \eqref{classical congruence subgroups}. Then every eigenform $f\in S_k(\Gamma_p)$ arises from a vector in $\pi_p^{C_p}$ by the above procedure. We thus obtain the formula
\begin{equation}\label{dimskreleq}
	\dim_{\C}S_k(\Gamma_p)= \sum\limits_{\Omega} \sum_{\pi  \in S_k(p,\Omega)} \dim\pi_p^{C_p}= \sum\limits_{\Omega} s_k(p,\Omega)\,d_{C_p,\Omega},
\end{equation}
which will be the basis for our determination of the quantities $s_k(p,\Omega)$. Here, the quantities $d_{C_p,\Omega}$ are defined in the notation section.

Note that the representations of type IVb and IVc are never unitary, and hence cannot occur as components of cuspidal, automorphic representations. Thus $s_k(p,\mathrm{IVb})=s_k(p,\mathrm{IVc})=0$. The one-dimensional representations of type IVd can also not occur as components of cuspidal, automorphic representations, so that $s_k(p,\mathrm{IVd})=0$. The family ${\rm Vc}$ is the same as the family ${\rm Vb}$ due to the fact that a representation of type ${\rm Vc}$ is just a representation of type ${\rm Vb}$ twisted by some suitable character; see \cite[(2.10)]{RobertsSchmidt2007}. So, in the following we will ignore ${\rm Vc}$, as it is subsumed under the ${\rm Vb}$ case.

\subsection{The quantities \texorpdfstring{$s_k(p,\Omega)$}{} and Arthur packets}\label{Section on Arthur packets}

Recall from \cite{Arthur2013} that every $\pi\in S_k(p,\Omega)$ appears in the discrete spectrum of $L^2(G(\Q)\backslash G(\A_{\Q}))$ inside a certain global Arthur packet. Using \cite[Theorem~1.5.2]{Arthur2013} for $\SO(5)$, we have
\begin{equation}\label{Arthur's main result}
 L^2_{\rm disc}(\SO(5,\Q)\backslash \SO(5,\A_{\Q}))\cong \bigoplus_{\psi \in \Psi_2(G)}\bigoplus_{\lbrace{\pi \in \Pi_{\psi}:\, \langle \cdot,\pi \rangle=\varepsilon_{\psi}\rbrace}} \pi,
\end{equation}
where $\psi$ is an Arthur parameter, and $\Pi_{\psi}$ is the corresponding global Arthur packet, consisting of certain equivalence classes of representations of $\SO(5,\A_{\Q})$. The quantities $\langle \cdot,\pi \rangle$ and $\varepsilon_{\psi}$ are characters of a centralizer group $\mathcal{S}_{\psi}\cong(\Z/2\Z)^t$. Since we identify the representations of $\SO(5,\A_{\Q})$ with the representations of $G(\A_{\Q})$ having trivial central character, a representation $\pi\in S_k(p,\Omega)$ appears in some global Arthur packet $\Pi_{\psi}$ on the right hand side of \eqref{Arthur's main result}.
The Arthur packets fall into six classes: the finite type {\bf(F)}, the general type {\bf(G)}, the Yoshida type {\bf(Y)}, and the types {\bf(P)}, {\bf(B)} and {\bf(Q)} consisting mostly of CAP representations (cuspidal associated to parabolics). Cuspidal representations cannot be finite-dimensional, so we will ignore the type {\bf(F)}. See \cite{Arthur2004,Schmidt2018} for more details about the Arthur packets for $\GSp(4)$.

Let $S_k^{(*)}(p,\Omega)$ be the set of those $\pi\in S_k(p,\Omega)$ that lie in an Arthur packet of type $(*)$. Evidently,
\begin{equation}\label{SkArthurdecompeq}
 S_k(p,\Omega)=S_k^{\rm (\mathbf{G})}(p,\Omega)\,\sqcup\,S_k^{\rm (\mathbf{Y})}(p,\Omega)\,\sqcup\,S_k^{\rm (\mathbf{P})}(p,\Omega)\,\sqcup\,S_k^{\rm (\mathbf{Q})}(p,\Omega)\,\sqcup\,S_k^{\rm (\mathbf{B})}(p,\Omega),
\end{equation}
so that
\begin{equation}\label{no. of cup form0}
 s_k(p,\Omega)=s_k^{\rm (\mathbf{G})}(p,\Omega)+s_k^{\rm (\mathbf{Y})}(p,\Omega)+s_k^{\rm (\mathbf{P})}(p,\Omega)+s_k^{\rm (\mathbf{Q})}(p,\Omega)+s_k^{\rm (\mathbf{B})}(p,\Omega),
\end{equation}
where $s_k^{(*)}(p,\Omega)=\#S_k^{(*)}(p,\Omega)$.

Not every representation can occur as a local component in every type of packet. In the last five columns of Table~\ref{rep with Arthur type} we indicate by $\bullet$ or $\circ$ the possible Arthur packet types in which a given local representation can occur. This information is based on the explicit determination of local Arthur packets given in \cite{Schmidt2020}.

A symbol ``$\circ$'' in the {\bf(Y)} column means that this representation can occur as a local component in a cuspidal automorphic representation $\pi$ inside an Arthur packet of Yoshida type, however such $\pi$ cannot be in $S_k(p,\Omega)$. The reason is that for discretely appearing $\pi\cong\otimes\pi_v$ in packets of type {\bf(Y)} the number of $\pi_v$'s that are non-generic has to be even; this is the concrete meaning of the multiplicity formula in \eqref{Arthur's main result} for Yoshida packets. Since for $\pi\in S_k(p,\Omega)$ the archimedean component is non-generic, the component $\pi_p$ must also be non-generic. Hence $s_k^{\mathbf{(Y)}}(p,\Omega)=0$ for the generic types $\Omega\in\{\mathrm{I},\mathrm{IIa},\mathrm{Va},\mathrm{VIa}\}$.

\begin{table}[H]
	\caption {Irreducible, admissible representations of $G(F)$ supported on the Borel subgroup.}
	\label{rep with Arthur type}
	\renewcommand{\arraystretch}{1.5}
	\renewcommand{\arraycolsep}{.3cm}
	\[ \begin{array}{ccccccccc} 
	\hline
	\multicolumn{2}{c}{\Omega}&\text{Representation}& \text{Tempered} &\textbf{(G)}&\textbf{(Y)}&\textbf{(P)}&\textbf{(Q)}&\textbf{(B)}\\
	\hline
	{\rm I} &&\chi_1\times\chi_2\rtimes\sigma\ \text{(irred.)}&\bullet&\bullet&\circ&&&\\
	\hline
	{\rm II} &\text{a}&\chi\St_{\GL(2)}\rtimes\sigma&\bullet&\bullet&\circ&&&\\
	\cline{2-9}
	&\text{b}&\chi\triv_{\GL(2)}\rtimes\sigma&&&&\bullet&&\\
	\hline
	{\rm III}&\text{a}&\chi\rtimes\sigma\St_{\GSp(2)} &\bullet&\bullet&&&&\\
	\cline{2-9}
	&\text{b}&\chi\rtimes\sigma\triv_{\GSp(2)}&&&&&\circ\\
	\hline
	{\rm IV}&\text{a}&\sigma\St_{\GSp(4)}&\bullet&\bullet&&&\\
	\cline{2-9}
	&\text{b}&L(\nu^2,\nu^{-1}\sigma\St_{\GSp(2)})&\multicolumn{6}{c}{\text{never unitary}}\\
	\cline{2-9}
	&\text{c}&L(\nu^{3/2}\St_{\GL(2)},\nu^{-3/2}\sigma)&\multicolumn{6}{c}{\text{never unitary}}\\
	\cline{2-9}
	&\text{d}&\sigma\triv_{\GSp(4)}&&&&&\\
	\hline
	{\rm V}&\text{a}&\delta([\xi,\nu\xi],\nu^{-1/2}\sigma)&\bullet&\bullet&\circ&&\\
	\cline{2-9}
	&\text{b}&L(\nu^{1/2}\xi\St_{\GL(2)},\nu^{-1/2}\sigma)&&&&\bullet&&\\
	\cline{2-9}
	&{\rm c}&L(\nu^{1/2}\xi\St_{\GL(2)},\xi\nu^{-1/2}\sigma) &&&&\circ&&\\
	\cline{2-9}
	&{\rm d}&L(\nu\xi,\xi\rtimes\nu^{-1/2}\sigma) &&&&&\circ&\circ\\
	\hline
	{\rm VI}&\text{a}&\tau(S,\nu^{-1/2}\sigma) &\bullet&\bullet&\circ&&&\\
	\cline{2-9}
	&{\rm b}&\tau(T,\nu^{-1/2}\sigma) &\bullet&\bullet&\bullet&\bullet&&\\
	\cline{2-9}
	&{\rm c}&L(\nu^{1/2}\St_{\GL(2)},\nu^{-1/2}\sigma) &&&&\bullet&\circ&\circ\\
	\cline{2-9}
	&{\rm d}&L(\nu,1_{F^\times}\rtimes\nu^{-1/2}\sigma) &&&&&\circ&\circ\\
	\hline
	\end{array}\]
\end{table}

We will prove in the next section that $s_k^{\mathbf{(Q)}}(p,\Omega)=s_k^{\mathbf{(B)}}(p,\Omega)=0$ for all $k$ and all $\Omega$; hence the $\circ$ in these columns of Table~\ref{rep with Arthur type}. By \eqref{no. of cup form0}, for all $k\geq1$,
\begin{equation}\label{no. of cup form1}
 s_k(p,\Omega)=s_k^{\rm (\mathbf{G})}(p,\Omega)+s_k^{\rm (\mathbf{P})}(p,\Omega)+s_k^{\rm (\mathbf{Y})}(p,\Omega),
\end{equation}
where $s_k^{\rm (\mathbf{Y})}(p,\Omega)=0$ unless $\Omega=\mathrm{VIb}$.

Note that Arthur packets of type {\bf(G)} are \emph{stable}, meaning one can switch within local $L$-packets and still retain the automorphic property. Most representations in Table \ref{rep with Arthur type} constitute singleton $L$-packets, except VIa and VIb, which constitute a $2$-element $L$-packet, and Va, which shares an $L$-packet with a non-generic supercuspidal. Hence $s_k^{\rm \mathbf{(G)}}(p, {\rm VIa})=s_k^{\rm \mathbf{(G)}}(p, {\rm VIb})$, and we denote this common number by $s_k^{\rm \mathbf{(G)}}(p, {\rm VIa/b})$:
\begin{equation}\label{IIIa/VI. VIa, VIb I}
 s_k^{\rm \mathbf{(G)}}(p, {\rm VIa/b})=s_k^{\rm \mathbf{(G)}}(p, {\rm VIa})=s_k^{\rm \mathbf{(G)}}(p, {\rm VIb}).
\end{equation}
A look at \cite[Table 3]{Schmidt200501} shows that, for each parahoric subgroup $H$, the dimensions of $H$-invariant vectors in a type IIIa representation is the same as the dimensions of $H$-invariant vectors for the $L$-packet VIa/b combined. (The reason for this is that both IIIa and ${\rm VIa}\oplus{\rm VIb}$ are of the form $\chi\rtimes\sigma\St_{\GSp(2)}$ for unramified characters $\chi$ and $\sigma$. The dimension of the space of $H$-invariant vectors in such a parabolically induced representation does not depend on $\chi$ and $\sigma$.) Hence our methods will not be able to determine the numbers $s_k^{\rm \mathbf{(G)}}(p, {\rm IIIa})$ and $s_k^{\rm \mathbf{(G)}}(p, {\rm VIa/b})$ separately, but we will be able to determine
\begin{equation}\label{IIIa/VI. VIa, VIb II}
 s_k(p, {\rm IIIa+VIa/b})\colonequals s_k^{\rm \mathbf{(G)}}(p, {\rm IIIa})+s_k^{\rm \mathbf{(G)}}(p, {\rm VIa/b}).
\end{equation}
Summarizing, we will compute the following $s_k(p,\Omega)$.
\begin{equation}
\label{mk that we compute}
\renewcommand{\arraystretch}{1.3}
\renewcommand{\arraycolsep}{.2cm}
\begin{array}{ cc} 
\toprule
\text{Tempered } &\text{Saito-Kurokawa type}\\[-1ex]
\text{representations} &\text{ representations}\\
\toprule
s_k(p,{\rm I})=s_k^{\rm \mathbf{(G)}}(p,{\rm I})&s_k(p,{\rm IIb})=s_k^{\rm \mathbf{(P)}}(p,{\rm IIb})\\
\toprule
s_k(p,{\rm IIa})=s_k^{\rm \mathbf{(G)}}(p,{\rm IIa})&s_k(p,{\rm Vb})=s_k^{\rm \mathbf{(P)}}(p,{\rm Vb})\\
\toprule
s_k(p, {\rm IIIa+VIa/b})=s_k^{\rm \mathbf{(G)}}(p, {\rm IIIa})+s_k^{\rm \mathbf{(G)}}(p, {\rm VIa/b})
&s_k^{\rm \mathbf{(P)}}(p,{\rm VIb})\\
\toprule
s_k(p,{\rm IVa})=s_k^{\rm \mathbf{(G)}}(p,{\rm IVa})&s_k(p,{\rm VIc})=s_k^{\rm \mathbf{(P)}}(p,{\rm VIc})\\
\toprule
s_k(p,{\rm Va})=s_k^{\rm \mathbf{(G)}}(p,{\rm Va})&\\
\toprule
s_k^{\rm \mathbf{(Y)}}(p,{\rm VIb})&\\
\bottomrule
\end{array}
\end{equation}

Representations of type {\bf(P)} and {\bf(Y)} are parametrized by automorphic representations of $\GL(2,\A_\Q)$. Hence, it is not difficult to express the $s_k$'s for the Saito-Kurokawa type representations as well as $s_k^{\rm \mathbf{(Y)}}(p,{\rm VIb})$ in terms of dimension formulas for certain elliptic modular forms; see Theorem~\ref{Theorem of I and IIb}, Theorem~\ref{Theorem of Saito-Kurokawa type}, and Theorem~\ref{relations of Yoshida type}. It then remains to determine the quantities $s_k(p,\Omega)$ for the five generic types I, IIa, IIIa+VIa/b, IVa, Va. On the other hand, we have the five congruence subgroups $\Gamma_p$ appearing in \eqref{classical congruence subgroups} (for $N=p$). It turns out that \eqref{dimskreleq} is a linear system with an invertible matrix, allowing us to calculate the $s_k(p,\Omega)$ from the $\dim_{\C}S_k(\Gamma_p)$. We do not have a good theoretical explanation for the non-singularity of this linear system.

\subsection{Packets of type {\bf(B)} and {\bf(Q)} are not Iwahori-spherical}\label{Section on B and Q types}
In this section we will prove that $s_k^{\mathbf{(Q)}}(p,\Omega)=s_k^{\mathbf{(B)}}(p,\Omega)=0$ for all $k$ and all $\Omega$. This is a consequence of a more general result about Arthur packets of type {\bf(Q)} or {\bf(B)}. This more general result also implies that certain spaces of Siegel modular forms of weight $1$ are zero.
\begin{proposition}\label{BQautoprop}
 Let $\pi\cong\otimes\pi_v$ be an automorphic representation of $\PGSp(4,\A_\Q)$ in an Arthur packet of type {\bf(Q)} or {\bf(B)}. Then $\pi_p$ is not Iwahori-spherical for at least one prime $p$.
\end{proposition}
\begin{proof}
Assume first that $\pi$ lies in an Arthur packet of type {\bf(Q)}. Recall from \cite{Schmidt2020} that one of the parameters entering into the definition of such a packet is a non-trivial quadratic character $\xi=\otimes\xi_v$ of $\Q^\times\backslash\A^\times$. There exists a prime $p$ for which the local character $\xi_p$ is ramified. Looking at Table 3 of \cite{Schmidt2020}, the description of the local packets of type {\bf(Q)}, we see that none of the possibilities for $\pi_p$ (corresponding to a ramified $\xi_p$) is an Iwahori-spherical representation.

The proof for $\pi$ being in a packet of type {\bf(B)} is similar. Such packets are parametrized by pairs $(\chi_1,\chi_2)$ of distinct, quadratic characters of $\chi_i=\otimes\chi_{i,v}$ of $\Q^\times\backslash\A^\times$. There exists a prime $p$ for which $\chi_{1,p}\chi_{2,p}$ is ramified. Then a look at Table 1 of \cite{Schmidt2020}, the description of the local packets of type {\bf(B)}, shows that none of the corresponding possibilities for $\pi_p$ is Iwahori-spherical.
\end{proof}

Recall, for example from \cite[Sect.~4.1]{Schmidt2017}, that vector-valued Siegel cusp forms of degree $2$ form spaces $S_{k,j}(\Gamma)$, where $k$ is a positive integer and $j$ is a non-negative integer. We have $S_{k,0}(\Gamma)=S_k(\Gamma)$. The following generalizes \cite[Theorem~6.1]{Ibukiyama2007}.
\begin{corollary}\label{BQautopropcor1}
 Let $N$ be a square-free positive integer. Then $S_{1,j}(\B(N))=0$ for any $j\geq0$.
\end{corollary}
\begin{proof}
Eigenforms in $S_{1,j}(\B(N))$ would have to originate from cuspidal, automorphic representations $\pi\cong\otimes\pi_v$ of $\PGSp(4,\A_\Q)$ for which $\pi_\infty$ is one of the lowest weight modules described in \cite[Sect.~2.4]{Schmidt2017}. Such lowest weight modules are non-tempered, and hence $\pi$ cannot be of type {\bf(G)}. (More precisely, weak archimedean estimates preclude the local parameter given in \cite[(3.17)]{Schmidt2017} to be that of the archimedean component of a cusp form on $\GL(4,\A)$.) For a similar reason, $\pi$ cannot be of type {\bf(Y)}. It can also not be of type {\bf(P)}, by the description of the local packets of type {\bf(P)} given in Table 2 of \cite{Schmidt2020}. It follows that $\pi$ must be of type {\bf(Q)} or {\bf(B)}. By Proposition~\ref{BQautoprop}, $\pi_p$ is not Iwahori-spherical for at least one prime $p$. But \emph{all} $\pi_p$ would have to be Iwahori-spherical in order to extract from $\pi$ an element of $S_{1,j}(\B(N))$.
\end{proof}
\begin{corollary}\label{BQautopropcor2}
 {\rm$s_k^{\textbf{(Q)}}(p,\Omega)=s_k^{\textbf{(B)}}(p,\Omega)=0$} for all $k$ and all $\Omega$. 
\end{corollary}
\begin{proof}
Since the elements of $S_k(p,\Omega)$ are, by definition, Iwahori-spherical at all finite places, this is an immediate consequence of Proposition~\ref{BQautoprop}.
\end{proof}

\section{Counting certain automorphic representations}\label{Section of Counting some cuspidal automorphic representations}
In Sect.~\ref{countPYsec} we give formulas for all $s_k^{\textbf{(P)}}(p,\Omega)$ and $s_k^{\textbf{(Y)}}(p,\Omega)$ (i.e., all the lifts), and in Sect.~\ref{countGsec} for all the $s_k^{\textbf{(G)}}(p,\Omega)$ (the non-lifts). In Sect.~\ref{newformssec} we express the dimensions of the spaces of newforms for some congruence subgroups of prime level in terms of the quantities we found. 

For the results in this section, we introduce some notation. The symbol $t = [t_0, t_1, . . . , t_{n-1};n]_k$ means that $t = t_i$ if $k \equiv i  \pmod{n}$. Then we define
\begin{equation}
\label{new notations}
\begin{split}
f_4(k)&=[k-2, -k+1, -k+2, k-1; 4]_k, \hspace{20ex} c_3(k)=[1,-1,0; 3]_k,\\
f_6(k)&=[k-3, -2k+2, -2k+4,k, k-1, k-2; 6]_k, \hspace{9ex}  \hat{c}_3(k)=[0,1,-1; 3]_k,\\
c_{12}(k)&=[1, 0, 0, -1, -1, -1, -1, 0, 0, 1, 1, 1; 12]_k, \hspace{14ex} c_4(k)=[1,0,0,-1; 4]_k,\\
c_6(k)&=[1,0, 0,-1, 0, 0; 6]_k, \hspace{34ex} c^{\prime}_4(k)=[1,-1,-1,1; 4]_k,\\
c^{\prime}_6(k)&=[0,1, 0, 0, -1, 0; 6]_k, \hspace{34ex} c_5(k)=[1,0,0,-1, 0;5]_k,\\
\hat{c}_6(k) &=[0,1, 1,0, -1, -1; 6]_k. 
\end{split}
\end{equation}
Let $\delta_{k,n}$ be the Kronecker symbol, i.e., $\delta_{k,n}=1$ if $k=n$ and $\delta_{k,n}=0$ otherwise. For any prime $p$, we define 
\begin{equation}
	\label{legendre}
\begin{split}
\left(\frac{-1}{p}\right)&=\begin{cases}
0&\quad\text{ if } p=2,\\
1&\quad\text{ if } p\equiv 1\pmod 4,\\
-1&\quad\text{ if } p\equiv 3\pmod 4.
\end{cases}\qquad\quad\ 
\left(\frac{-3}{p}\right)=\begin{cases}
0&\quad\text{ if } p=3,\\
1&\quad\text{ if } p\equiv 1\pmod 3,\\
-1&\quad\text{ if } p\equiv 2\pmod 3.
\end{cases}
\end{split}
\end{equation}
Note that $\left(\frac{\cdot}{p}\right)$ is not the usual Legendre symbol as we define $\left(\frac{-1}{2}\right)=0$.
\subsection{The representation types \texorpdfstring{${\rm I, IIb,  Vb, VIb, VIc}$}{PDFstring}}\label{countPYsec}
 In this section, we compute the numbers $s_k(p,\Omega)$ in \eqref{mk that we compute} when the representations in $S_k(p,\Omega)$ are of Saito-Kurokawa type {\bf(P)} or of Yoshida type {\bf(Y)}; these are two kinds of liftings from elliptic cuspidal automorphic representations. 

 We denote by $S_{k}^{\pm}(\SL(2,\Z))$ the space spanned by the eigenforms in $S_k(\SL(2,\Z))$ which have the sign $\pm 1$ in the functional equation of their $L$-functions. For the theorems in this section, we use the following standard formula
 \begin{equation}
 \label{dim Sk(SL2Z)}
 \begin{split}
 \dim_{\C}  S_{k}(\SL_2(\Z))=&\frac{k-1}{2^23}+\frac{1}{2^2} (-1)^{\frac k2}+\frac 13\left(c_3(k)+\hat{c}_3(k)\right)-\frac 12 +\delta_{k,2}.
 \end{split}
 \end{equation}
 
 \begin{theorem}\label{Theorem of I and IIb} For $k\ge 1$ we have
	\begin{align}
		\label{General formula for I and IIb eq1}s_k(p,{\rm I})
	=&\frac{1}{2^{7}3^{3}5}(k-2)(k -1)(2k-3)+\frac{7(-1)^k}{2^{7}3^{2}}(k-2)(k-1)+\frac{5}{2^43}+\delta_{k, 3}-\delta_{k, 2}-\frac{47}{2^{7}3^{3}}(2k -3)\nonumber\\
+&\frac{61}{2^7}(-1)^k-\frac{13}{2^23^3}\hat{c}_3(k)-\frac{1}{2\cdot 3}c_3(k)
+\frac{1}{2^{5}3}f_4(k)-\frac{1}{2^{3}}c^{\prime}_4(k)+\frac{1}{2^{3}}c_4(k)+\frac{1}{5}c_5(k)\\
+&\frac{1}{2^23^3}f_6(k)+\frac{1}{2^23}\hat{c}_6(k)+\frac{1}{3^2}c_6(k)+\frac{1}{2^23}c_{12}(k)-\frac{2k-3}{2^4}(-1)^k-\frac{1}{2\cdot 3}\hat{c}_3(k)(-1)^k.\nonumber\\
				\label{General formula for I and IIb eq2}s_k(p,{\rm IIb})=&\begin{cases}
	\displaystyle\frac{2k-3}{2^23}-\frac{3}{2^2}+\frac{1}3\hat{c}_3(k)+\delta_{k, 2}&\text{if $k$ is even},\\[2ex]
	0&\text{if $k$ is odd}.
	\end{cases}
	\end{align}
	Equivalently,
 \begin{align}\label{Theorem of I and IIb eq1}
  \sum\limits_{\substack{k \ge 1}}s_k(p,{\rm I})t^k
  =&\frac{t^{35}+1}{(1-t^4)(1-t^6)(1-t^{10})(1-t^{12})}-\frac{1}{(1-t^4)(1-t^6)}-\frac{t^{10}}{(1-t^2)(1-t^{6})},\\
  \label{Theorem of I and IIb eq2}\sum\limits_{\substack{k \ge 1 }}s_k(p,{\rm IIb})t^k=&\frac{t^{10}}{(1-t^2)(1-t^{6})}.
 \end{align}
\end{theorem}
\begin{proof}\label{proof of Theorem of I and IIb}
By \cite[Table~3]{Schmidt200501}, the spherical representations in Table~\ref{rep with Arthur type} are I, IIb, IIIb, IVd, Vd and VId. Hence, by \eqref{dimskreleq}, and observing that one-dimensionals are irrelevant, we have
\begin{equation}\label{Sp4Zdimeq}
 \dim_{\C} S_k(\Sp(4,\Z))=s_k(p,{\rm I})+s_k(p,{\rm IIb})+s_k(p,{\rm IIIb})+s_k(p,{\rm Vd})+s_k(p,{\rm VId}).
\end{equation}
By Corollary~\ref{BQautopropcor2}, $s_k(p,{\rm IIIb})=s_k^{\textbf{(Q)}}(p,{\rm IIIb})=0$, and similarly, $s_k(p,{\rm Vd})=s_k(p,{\rm VId})=0$. Thus
\begin{equation}\label{Sp4Zdimeq2}
 \dim_{\C} S_k(\Sp(4,\Z))=s_k(p,{\rm I})+s_k(p,{\rm IIb}).
\end{equation}
The eigenforms constructed from the representations in $S_k(p,{\rm IIb})$ are precisely the full-level Saito-Kurokawa liftings. By \cite{EichlerZagier1985}, the space spanned by these is the image of an injective map
	\begin{equation*}
	S_{2k-2}^{-}(\SL(2,\Z)) \hookrightarrow S_k(\Sp(4,\Z)).
	\end{equation*}
Since the sign in the functional equation of the eigenforms in $S_{2k-2}(\SL(2,\Z))$ is $(-1)^{k-1}$, we get
\begin{equation}\label{m_IIb}
 s_k(p,{\rm IIb})=\dim_{\C} S_{2k-2}^{-}(\SL(2,\Z))=
 \begin{cases}
   \dim_{\C} S_{2k-2}(\SL(2,\Z))\ \qquad &\text{if } k \text{ is even},\\
   0&\text{if } k \text{ is odd}.
 \end{cases}
\end{equation}
Then \eqref{General formula for I and IIb eq2} and \eqref{Theorem of I and IIb eq2} follow from \eqref{dim Sk(SL2Z)}. 
Furthermore, using \eqref{Sp4Zdimeq2} and $\dim_{\C}S_{k}(\Sp(4,\Z))$ from \cite[Theorem~6-2]{Hashimoto1983} (which is also implicit in Sect.~3 of \cite{Igusa1964}), we obtain \eqref{General formula for I and IIb eq1}. Note that the dimension formula for $\Sp(4,\Z))$ in \cite[Theorem~6-2]{Hashimoto1983} does not work for $k=3$; we add the term $\delta_{k, 3}$ so that the formula works for all $k$.
\end{proof}
The following lemma is useful for finding the quantities $s_k(p,\Omega)$ for the representations of Saito-Kurokawa type and Yoshida type. This result can be derived from the work of Yamauchi \cite{Yamauchi1973}, and it is explicitly given in \cite[Theorem 2.2]{Martin2018}. Here, $S_{k}^{\rm new}(\Gamma_{0}^{(1)}(p))$ is the new subspace of weight $k$ elliptic cusp forms on the congruence subgroup $\Gamma_{0}^{(1)}(p)$ of $\SL(2,\Z)$. The plus and minus spaces $S_{k}^{\pm,\rm new}(\Gamma_{0}^{(1)}(p))$ are the space spanned by the eigenforms in $S_{k}^{\rm new}(\Gamma_{0}^{(1)}(p))$ which have the sign $\pm 1$ in the functional equation of their $L$-functions. 

\begin{lemma}\label{theorem for +,- new space}
	For any even integer $k\geq 2$ and $p\ge5$,
	\begin{equation*}
	\dim_{\C} S_{k}^{\pm,\rm new}(\Gamma_{0}^{(1)}(p))=\frac 12\dim_{\C} S_{k}^{\rm new}(\Gamma_{0}^{(1)}(p))\pm \frac12 \left(\frac12 hb-\delta_{k,2}\right),
	\end{equation*}
	where $h$ is the class number of $\Q(\sqrt{-p})$ and  
	\begin{equation}
	\label{b}
	b=\begin{cases} 
	1 & \text{if } p \equiv 1 \pmod 4\\
	2 & \text{if } p \equiv 7 \pmod 8\\
	4 & \text{if } p \equiv 3 \pmod 8,\\
	\end{cases} 
	\hspace{0.5in} \text{and}	\hspace{0.5in}
	\delta_{k,2}=
	\begin{cases}
	1 & \text{if } k=2\\
	0 & \text{if } k\neq 2.
	\end{cases}
	\end{equation}
	For $k>2$,
	\begin{align*}
	\dim_{\C} S_{k}^{\pm,\rm new}(\Gamma_{0}^{(1)}(2))&=\frac 12\dim_{\C} S_{k}^{\rm new}(\Gamma_{0}^{(1)}(2))\pm
	\begin{cases} 
	\frac12 & \text{if } k \equiv 0,2 \pmod 8,\\
	0 & \text{else}.
	\end{cases}\\
	\dim_{\C} S_{k}^{\pm,\rm new}(\Gamma_{0}^{(1)}(3))&=\frac 12\dim_{\C} S_{k}^{\rm new}(\Gamma_{0}^{(1)}(3))\pm
	\begin{cases} 
	 \frac12 & \text{if } k \equiv 0,2,6,8 \pmod {12},\\
	0 & \text{else}.
	\end{cases}
	\end{align*}
\end{lemma}
 \noindent For any even integer $k\geq 2$, we have the following well-known results (for details see \cite[Sect.~3.5]{DiamondShurman2005}),
\begin{equation}
\label{dim Sk(Gamma0)}
\begin{split}
&\dim_{\C}  S_{k}(\Gamma_0^{(1)}(p))=\frac{k-1}{2^23}(p+1)+\frac{(-1)^{\frac k2}}{2^2} \left(1+\left(\frac{-1}{p}\right)\right)+\frac{c_3(k)+\hat{c}_3(k)}{3} \left(1+\left(\frac{-3}{p}\right)\right)-1+\delta_{k,2},\\
&\dim_{\C} S_{k}^{\text{new}}(\Gamma_{0}^{(1)}(p))=\dim_{\C} S_{k}(\Gamma_{0}^{(1)}(p))-2 \dim_{\C} S_{k}(\SL(2,\Z)). 
\end{split}
\end{equation}

\begin{theorem}\label{Theorem of Saito-Kurokawa type}
	Let $p\geq5$ be a prime. Suppose $h$ is the class number of $\Q(\sqrt{-p})$ and $b$ is defined as in \eqref{b}. 
 \begin{enumerate}
  \item For $k\geq2$,
   $$
    s_k(p,{\rm Vb})=\begin{cases}
                     \displaystyle\frac{2k-3}{2^33}(p-1)+\frac{1-\big(\frac{-1}{p}\big)}{2^3}-\frac{\hat{c}_3(k)\big(1-\big(\frac{-3}{p}\big)\big)}{2\cdot 3}-\frac{bh}{2^2}&\text{if $k$ is even},\\[2ex]
                     0&\text{if $k$ is odd}.
                    \end{cases}
   $$
   Equivalently,
   $$
    \sum\limits_{\substack{k \geq 2}}s_k(p,{\rm Vb})t^k=\left[\frac{(1+3t^2)(p-1)}{2^33\left(1-t^2\right)^2}+\frac{1-\big(\frac{-1}{p}\big)}{2^3 \left(1-t^2\right)}+\frac{1-\big(\frac{-3}{p}\big)}{2\cdot 3(1+t^2+t^4)}-\frac{bh}{2^2(1-t^2)}\right]t^2.
   $$
  \item For $k\geq2$,
   $$
    s_k^{\text{\bf{(P)}}}(p,{\rm VIb})=\begin{cases}
                     \displaystyle\frac{2k-3}{2^33}(p-1)+\frac{1-\big(\frac{-1}{p}\big)}{2^3}-\frac{\hat{c}_3(k)\big(1-\big(\frac{-3}{p}\big)\big)}{2\cdot 3}+\frac{bh}{2^2}-\delta_{k,2}&\text{if $k$ is even},\\[2ex]
                     0&\text{if $k$ is odd}.
                    \end{cases}
   $$
   Equivalently,
   $$
    \sum\limits_{\substack{k \geq 2}}s_k^{\text{\bf{(P)}}}(p,{\rm VIb})t^k=\left[\frac{(1+3t^2)(p-1)}{2^33\left(1-t^2\right)^2}+\frac{1-\big(\frac{-1}{p}\big)}{2^3 \left(1-t^2\right)}+\frac{1-\big(\frac{-3}{p}\big)}{2\cdot3(1+t^2+t^4)}+\frac{bh}{2^2(1-t^2)}-1\right]t^2.
   $$
  \item For $k\geq2$,
   $$
    s_k(p,{\rm VIc})=\begin{cases}
                     0&\text{if $k$ is even},\\[2ex]
                     \displaystyle\frac{2k-3}{2^33}(p-1)-\frac{1-\big(\frac{-1}{p}\big)}{2^3}-\frac{\hat{c}_3(k)\big(1-\big(\frac{-3}{p}\big)\big)}{2\cdot3}-\frac{bh}{2^2}&\text{if $k$ is odd}.
                    \end{cases}
   $$
   Equivalently,
   $$
    \sum\limits_{\substack{k \geq 2}}s_k(p,{\rm VIc})t^k=\left[\frac{(3+t^2)(p-1)}{2^33\left(1-t^2\right)^2}-\frac{1-\big(\frac{-1}{p}\big)}{2^3 \left(1-t^2\right)}+\frac{\big(1-\big(\frac{-3}{p}\big)\big)t^2}{2\cdot3(1+t^2+t^4)}-\frac{bh}{2^2(1-t^2)}\right]t^3.
   $$
 \end{enumerate}
\end{theorem}
\begin{proof}
	It follows from \cite[Sect.~1]{Schmidt2007} that
	\begin{equation}\label{relations of Saito-Kurokawa type}
	\begin{split}
	s_k(p,{\rm Vb})=\begin{cases}
  \dim_{\C} S_{2k-2}^{-,\rm new}(\Gamma_{0}^{(1)}(p)) \qquad &\text{if } k \text{ is even},\\
   0&\text{if } k \text{ is odd}.
 \end{cases}\\
 	s_k^{\textbf{(P)}}(p,{\rm VIb})= \begin{cases}
 \dim_{\C} S_{2k-2}^{+,\rm new}(\Gamma_{0}^{(1)}(p))\ \qquad &\text{if } k \text{ is even},\\
 0&\text{if } k \text{ is odd}.
 \end{cases}\\
 	s_k(p,{\rm VIc})=\begin{cases}
 0&\text{if } k \text{ is even},\\
 \dim_{\C} S_{2k-2}^{-,\rm new}(\Gamma_{0}^{(1)}(p))\ \qquad &\text{if } k \text{ is odd}.
 \end{cases}\\
	\end{split}
	\end{equation}  
Then, using Lemma~\ref{theorem for +,- new space}, \eqref{dim Sk(SL2Z)} and \eqref{dim Sk(Gamma0)}, the formulas for $s_k^{\rm \mathbf{(P)}}(p,{\rm VIb})$, $s_k(p,{\rm VIc})$ and $s_k(p,{\rm Vb})$ follow from straightforward calculations.
\end{proof}
\noindent For the next theorem, we define
\begin{align}
C(p)=&\frac{p-1}{2^33}+\frac{1-\big(\frac{-1}{p}\big)}{2^3}+\frac{1-\big(\frac{-3}{p}\big)}{2\cdot 3}-\frac{1}{2}
\end{align}
for any prime $p\geq5$.

\begin{theorem}\label{relations of Yoshida type}
	Let $p\geq5$ be a prime. Suppose $h$ is the class number of $\Q(\sqrt{-p})$ and $b$ is defined as in \eqref{b}. 
For $k\geq 2$,
\begin{align*}
s_k^{\rm \mathbf{(Y)}}(p,{\rm VIb})&=C(p)\cdot \left(\frac{2k-3}{2^23}(p-1)+(-1)^k\frac{1-\big(\frac{-1}{p}\big)}{2^2}-\frac{\hat{c}_3(k)\big(1-\big(\frac{-3}{p}\big)\big)}3-\delta_{k,2}\right)\\
&+\frac{2-bh}{2^2}\delta_{k,2}+(-1)^k\frac{b^2h^2-2bh}{2^3}.
\end{align*}
 Equivalently,
 \begin{align*}
&\sum_{k\geq 2}s_k^{\rm \mathbf{(Y)}}(p,{\rm VIb})t^k\\
=&\left[C(p)\left(\frac{1+t}{2^23(1-t)^2}(p-1)+\frac{1-\big(\frac{-1}{p}\big)}{2^2(1+t)}+\frac{(1+t)\big(1-\big(\frac{-3}{p}\big)\big)}{3(1+t+t^2)}-1\right)+\frac{2-bh}{2^2}
+\frac{b^2h^2-2bh}{2^3(1+t)}\right]t^2. 
 \end{align*}
\end{theorem}
\begin{proof}
	In order to compute $s_k^{\textbf{(Y)}}(p,{\rm VIb})$ we look at the Yoshida lifting. This lifting associates a holomorphic Siegel modular form $F\in S_{k}^{\rm new}(\Gamma_0(p))$ to two eigenforms $f \in S_{2k-2}^{\rm new}({\rm \Gamma_0^{(1)}}(p))$ and $g \in S_{2}^{\rm new}({\rm \Gamma_0^{(1)}}(p))$; see \cite[Proposition~3.1]{SahaSchmidt2013}. Let $\pi_f=\otimes_{v \le \infty} \pi_{f,v}$ and $\pi_g=\otimes_{v \le \infty} \pi_{g,v}$ be the automorphic representations of $\PGL(2, \A_\Q)$ attached to $f$ and $g$, respectively. Both local representations $\pi_{f,\, p}$ and $\pi_{g,\, p}$ are either the Steinberg representation $\rm St_{\GL(2)}$ or its non-trivial unramified twist $\xi\mathrm{St}_{\GL(2)}$; here $\xi$ is the unique non-trivial unramified quadratic character of $\Q_p^\times$. In order to produce a Yoshida lifting $\pi_F=\otimes_{p \le \infty} \pi_{F,\, p}$ with Iwahori-fixed vectors at $p$, we must have $\pi_{f,\, p}=\pi_{g,\, p}$ by \cite[(16)]{SahaSchmidt2013}. In this case $\pi_{F,\, p}$ is of type $\rm VIb$. More precisely, if $\pi_{f,\, p}=\pi_{g,\, p}=\St_{\GL(2)}$ (and hence the local root numbers at $p$ are $-1$), then $\pi_{F,p}=\tau(S,\nu^{-1/2})$, and if $\pi_{f,\, p}=\pi_{g,\, p}=\xi\St_{\GL(2)}$ (and hence the local root numbers at $p$ are $+1$), then $\pi_{F,p}=\tau(S,\nu^{-1/2}\xi)$. Since the archimedean signs are $(-1)^{k-1}$ for eigenforms in $S_{2k-2}^{\rm new}({\rm \Gamma_0^{(1)}}(p))$, and $-1$ for eigenforms in $S_2^{\rm new}({\rm \Gamma_0^{(1)}}(p))$, it follows that
	\begin{equation}\label{general formula for Yoshida type}
	s_k^{\rm \textbf{(Y)}}(p,{\rm VIb})=
	\begin{cases*}
	\!\begin{aligned}
	& \dim_{\C} S_{2k-2}^{+,\rm new}({\rm \Gamma_0^{(1)}}(p))\times \dim_{\C} S_2^{+,\rm new}({\rm \Gamma_0^{(1)}}(p))\\
	&+\dim_{\C} S_{2k-2}^{-,\rm new}({\rm \Gamma_0^{(1)}}(p))\times \dim_{\C} S_2^{-,\rm new}({\rm \Gamma_0^{(1)}}(p))
	\end{aligned} &\text{if} $k$ \text{ is even},\\[0.1 cm]
	\!\begin{aligned}
	&\dim_{\C} S_{2k-2}^{-,\rm new}({\rm \Gamma_0^{(1)}}(p))\times \dim_{\C} S_2^{+,\rm new}({\rm \Gamma_0^{(1)}}(p))\\
	&+ \dim_{\C} S_{2k-2}^{+,\rm new}({\rm \Gamma_0^{(1)}}(p))\times \dim_{\C} S_2^{-,\rm new}({\rm \Gamma_0^{(1)}}(p))
	\end{aligned}  &\text{if} $k$ \text{ is odd}.
	\end{cases*}
	\end{equation}
	Then, using \eqref{dim Sk(SL2Z)}, \eqref{dim Sk(Gamma0)}, \eqref{relations of Saito-Kurokawa type}, \eqref{general formula for Yoshida type} and Lemma~\ref{theorem for +,- new space},  we obtain the general formula and hence the generating function for  $s_k^{\rm \mathbf{(Y)}}(p,{\rm VIb})$. 
\end{proof}

\subsection{The representation types \texorpdfstring{${\rm IIa, IIIa+VIa/b, IVa, Va}$}{PDFstring}}\label{countGsec}
In this section, we compute the generating functions $\sum_{k\geq 3}s_k(p,\Omega)t^k$ for the representation types $\Omega\in\{{\rm IIa, IIIa+VIa/b, IVa, Va}\}$. Note that the Arthur packets contributing to such  $S_k(p,\Omega)$ are necessarily of type $\textbf{(G)}$.

We will use dimension formulas for the spaces $S_k(\Gamma)$, where $\Gamma$ is one of $\mathrm{K}(p)$, $\Gamma_0(p)$, $\Gamma_0'(p)$ or $B(p)$. Many authors have contributed such dimension formulas. We summarize the sources of the formulas we will use in Table~\ref{historytable}, without claim to historical completeness.

\begin{table}
	\caption{History of dimension formulas for congruence subgroups of Iwahori-type. Earlier references appear above later references.}
	\label{historytable}
$$
\renewcommand{\arraystretch}{1.2}
\renewcommand{\arraycolsep}{.5cm}
 \begin{array}{cclll}
 	\toprule
   &&p=2&p=3&p\geq5\\
   \midrule
    \mathrm{K}(p)&k=1&\text{\cite[Thm.~6.1]{Ibukiyama2007}}&\text{\cite[Thm.~6.1]{Ibukiyama2007}}&\text{\cite[Thm.~6.1]{Ibukiyama2007}}\\
   \cmidrule{2-5}
    &k=2&\text{\cite[Sect.~1]{Ibukiyama1984}}&\text{\cite[Sect.~5.3]{Ibukiyama2018}}&\text{\cite{PoorYuen2015}}\ (p < 600)\\
   \cmidrule{2-5}
    &k=3&\text{\cite[Thm.~2.1]{Ibukiyama2007}}&\text{\cite[Thm.~2.1]{Ibukiyama2007}}&\text{\cite[Thm.~2.1]{Ibukiyama2007}}\\
   \cmidrule{2-5}
    &k=4&\text{\cite[Sect.~2.4]{Ibukiyama2007}}&\text{\cite[Sect.~2.4]{Ibukiyama2007}}&\text{\cite[Sect.~2.4]{Ibukiyama2007}}\\
   \cmidrule{2-5}
    &k\geq5&\text{\cite[Thm.~4]{Ibukiyama1985}}&\text{\cite[Thm.~4]{Ibukiyama1985}}&\text{\cite[Thm.~4]{Ibukiyama1985}}\\
   \midrule
    \Gamma_0(p)&k=1&\text{\cite[Thm.~6.1]{Ibukiyama2007}}&\text{\cite[Thm.~6.1]{Ibukiyama2007}}&\text{\cite[Thm.~6.1]{Ibukiyama2007}}\\
   \cmidrule{2-5}
    &k=2&\text{\cite[Sect.~1]{Ibukiyama1984}}&\text{\cite[Sect.~5.3]{Ibukiyama2018}}\\
   \cmidrule{2-5}
    &k=3&\text{\cite[Thm.~2.2]{Ibukiyama2007}}&\text{\cite[Thm.~2.2]{Ibukiyama2007}}&\text{\cite[Thm.~2.2]{Ibukiyama2007}}\\
   \cmidrule{2-5}
    &k=4&\text{\cite[Cor.~4.12]{Tsushima1997}}&\text{\cite[Cor.~4.12]{Tsushima1997}}&\text{\cite[Cor.~4.12]{Tsushima1997}}\\
    &&\text{\cite[Sect.~2.4]{Ibukiyama2007}}&\text{\cite[Sect.~2.4]{Ibukiyama2007}}&\text{\cite[Sect.~2.4]{Ibukiyama2007}}\\
   \cmidrule{2-5}
    &k\geq5&\text{\cite[Cor.~4.12]{Tsushima1997}}&\text{\cite[Thm.~7-1]{Hashimoto1983}}&\text{\cite[Thm.~7-1]{Hashimoto1983}}\\
    &&\text{\cite[Thm.~7.4]{Wakatsuki2012}}&\text{\cite[Cor.~4.12]{Tsushima1997}}&\text{\cite[Cor.~4.12]{Tsushima1997}}\\
    &&&\text{\cite[Thm.~7.4]{Wakatsuki2012}}&\text{\cite[Thm.~7.4]{Wakatsuki2012}}\\
   \midrule
    \Kl(p)&k=1&\text{\cite[Thm.~6.1]{Ibukiyama2007}}&\text{\cite[Thm.~6.1]{Ibukiyama2007}}&\text{\cite[Thm.~6.1]{Ibukiyama2007}}\\
   \cmidrule{2-5}
    &k=2&\text{\cite[Sect.~1]{Ibukiyama1984}}&\text{\cite[Sect.~5.3]{Ibukiyama2018}}\\
   \cmidrule{2-5}
    &k=3&\text{\cite[Thm.~2.4]{Ibukiyama2007}}&\text{\cite[Thm.~2.4]{Ibukiyama2007}}&\text{\cite[Thm.~2.4]{Ibukiyama2007}}\\
   \cmidrule{2-5}
    &k=4&\text{\cite[Sect.~2.4]{Ibukiyama2007}}&\text{\cite[Sect.~2.4]{Ibukiyama2007}}&\text{\cite[Sect.~2.4]{Ibukiyama2007}}\\
   \cmidrule{2-5}
    &k\geq5&\text{\cite[Thm.~A.1]{Wakatsuki2013}}&\text{\cite[Thm.~A.1]{Wakatsuki2013}}&\text{\cite[Thm.~3.3]{HashimotoIbukiyama1985}}\\
    &&&&\text{\cite[Thm.~A.1]{Wakatsuki2013}}\\
   \midrule
    \B(p)&k=1&\text{\cite[Thm.~6.1]{Ibukiyama2007}}&\text{\cite[Thm.~6.1]{Ibukiyama2007}}&\text{\cite[Thm.~6.1]{Ibukiyama2007}}\\
   \cmidrule{2-5}
    &k=2&\text{\cite[Sect.~1]{Ibukiyama1984}}&\text{Prop.~\ref{k2prop}}\\
   \cmidrule{2-5}
    &k=3&\text{\cite[Thm.~2.3]{Ibukiyama2007}}&\text{\cite[Thm.~2.3]{Ibukiyama2007}}&\text{\cite[Thm.~2.3]{Ibukiyama2007}}\\
   \cmidrule{2-5}
    &k=4&\text{\cite[Sect.~2.4]{Ibukiyama2007}}&\text{\cite[Sect.~2.4]{Ibukiyama2007}}&\text{\cite[Sect.~2.4]{Ibukiyama2007}}\\
   \cmidrule{2-5}
    &k\geq5&\text{\cite[Thm.~A.2]{Wakatsuki2013}}&\text{\cite[Thm.~A.2]{Wakatsuki2013}}&\text{\cite[Thm.~3.2]{HashimotoIbukiyama1985}}\\
    &&&&\text{\cite[Thm.~A.2]{Wakatsuki2013}}\\
    \bottomrule
 \end{array}
$$
\end{table}

In the following theorems, let $b$ and $h$ be as in Lemma~\ref{theorem for +,- new space}. We will also use the quantities defined in \eqref{new notations} and \eqref{legendre}.

\begin{theorem}\label{Theorem of IIa1}
	Let $p\geq5$ be a prime and $k\ge 3$. Then	
	\begin{equation*}
	\begin{split}
		s_k(p,{\rm IIa})=
	&\frac{p^2-1}{2^73^35}(k-2) (k-1) (2 k-3)+\frac{-4 \big(\frac{-3}{p}\big)-3 \big(\frac{-1}{p}\big)+p-3}{2^33}+\frac{bh}{2^2}-\delta_{k, 3}\\
	+&\frac{16 (p+3) \big(\frac{-3}{p}\big)+9 (p+4) \big(\frac{-1}{p}\big)-84 p+119}{2^73^3}(2k-3)\\
	+&\frac{\left(16 \big(\frac{-3}{p}\big)-p+12\right) \left(\big(\frac{-1}{p}\big)-1\right)+3 (p-49)}{2^73}(-1)^k+\frac{(-1)^k(2k-3)}{2^33}\\
	+&\frac{\left(\big(\frac{-3}{p}\big)+1\right) \left(9 \big(\frac{-1}{p}\big)+p-6\right)-4 (p-8)}{2^33^3}\hat{c}_3(k)+\frac{(-1)^k\hat{c}_3(k)}{2\cdot 3}\\
	-&\frac{c_4(k)}{2^2} \begin{cases}
	0&p\equiv 1,7\pmod 8,\\
	1&p\equiv 3,5\pmod 8,
	\end{cases}
	-\frac{c_5(k)}{5}\begin{cases}
	1&p=5,\\
	0&p\equiv 1,4\pmod {5},\\
	2&p\equiv 2,3\pmod {5}.
	\end{cases}
	\end{split}
\end{equation*}
Equivalently,
	\begin{equation*}
	\begin{split}
\displaystyle	&\sum_{k\geq 3}s_k(p,{\rm IIa})t^k\\
		=&\left[\frac{\left(p^2-1\right)(1+t)}{2^63^25 (1-t)^4}-\frac{(p-1)(5 + 13 t + 17 t^2 + 12 t^3 + 12 t^4)}{2^53^2 (1-t^2)(1-t^3)}+\frac{t^{7}}{(1 - t^2) (1-t^6)}+\frac{bh}{2^2(1-t)}-1\right.\\
		+&\left.\left(\frac{p(1+t)}{2^33^2 (1-t)(1-t^3)}+\frac{-3 + t + 4 t^3}{2^33(1-t)(1-t^3)}\right)\left(\left(\frac{-3}{p}\right)-1\right)\right.\\
		+&\left.\left(\frac{p}{2^53 (1-t) (1-t^2)}+\frac{-4 + 2 t - t^2 + 3 t^3 + 3 t^4}{2^33(1-t^2)(1-t^3)}\right)\left(\left(\frac{-1}{p}\right)-1\right)
		-\frac{\left(\big(\frac{-3}{p}\big)-1\right)\left(\big(\frac{-1}{p}\big)-1\right)}{2^33 (1+t)(1+t+t^2)}
		\right]t^3\\
			+&\frac{t^3}{2^2 (1+t)(1+t^2)} \begin{cases}
		0&p\equiv 1,7\pmod 8,\\
		1&p\equiv 3,5\pmod 8,
		\end{cases}
		+\frac{(1+t)t^3}{5 (1 + t + t^2 + t^3 + t^4)}\begin{cases}
		1&p=5,\\
		0&p\equiv 1,4\pmod {5},\\
		2&p\equiv 2,3\pmod {5}.
		\end{cases}
	\end{split}
\end{equation*}
\end{theorem}

\begin{theorem}\label{Theorem of IIIa1}
	Let $p\geq5$ be a prime and $k\ge 3$. Then	
	\begin{align*}
&s_k(p,{\rm IIIa+VIa/b})\\
=\ &\frac{(p-1)(p^2+p+2)}{2^83^35}(k-2) (k-1) (2 k-3)+\frac{3\big(\frac{-1}{p}\big)-p-2}{2^43}-\frac{bh}{2^3}-\frac{b^2h^2-2bh}{2^4}(-1)^k-\delta_{k, 3}\\
+&\frac{7 (p-1)(p+3)(-1)^k}{2^83^2}(k-2) (k-1)-\frac{(p-1) \left(-32 \big(\frac{-3}{p}\big)-27 \big(\frac{-1}{p}\big)+12 p-97\right)}{2^83^3}(2k-3)\\ 
-&\frac{\left(32 \big(\frac{-3}{p}\big)-5 p-3\right) \left(9 \big(\frac{-1}{p}\big)-17\right)-40 (p+7)}{2^83^3}(-1)^k -\frac{p-1}{2^33} (-1)^k (2 k-3) +\frac{1-\big(\frac{-3}{p}\big)}{2\cdot 3} c_3(k)\\
+&
\frac{(p+5) \left(1-\big(\frac{-3}{p}\big)\right)}{2^23^3}\hat{c}_3(k)+\frac{1-\big(\frac{-3}{p}\big)}{2^23}\hat{c}_3(k)(-1)^k  
+\frac{(p+1) \big(\frac{-1}{p}\big)+p-3}{2^63} f_4(k)+\frac{1-\big(\frac{-1}{p}\big)}{2^3} c^{\prime}_4(k)\\
+&\frac{(p+1) \big(\frac{-3}{p}\big)+p-3}{2^33^3} f_6(k)+\frac{2\left(\big(\frac{-3}{p}\big)-2\right)}{3^3}c_6(k)+\frac{5 \big(\frac{-3}{p}\big)-13}{2^23^3} \hat{c}_6(k)+\frac{2\left(\big(\frac{-3}{p}\big)+1\right)}{3^3}c^{\prime}_6(k)\\
+&\frac{\left(\big(\frac{-3}{p}\big)+1\right)\left(\big(\frac{-1}{p}\big)+1\right)-4}{2^33} c_{12}(k)
-\frac{c_4(k)}{2^2}\begin{cases}
	1&p\equiv 7\pmod{8},\\
	0&\text{otherwise},
\end{cases}-\frac{c_5(k)}{5}\begin{cases}
	1&p=5,\\
	0&p\equiv 1\hspace{0.05in}\pmod{5},\\
	1&p\equiv 2, 3 \hspace{-0.1in}\pmod{5},\\
	2&p\equiv 4\hspace{0.05in}\pmod{5}.
\end{cases}
\end{align*}

\vspace*{-0.2in}	
Equivalently,
\begin{align*}
&\sum_{k\geq 3}s_k(p,{\rm IIIa+VIa/b})t^k\\
=&\left[\frac{(p-1)(p^2+p+2) (1+t)}{2^73^25 (1-t)^4}-\frac{(p-1)(p+3)(13 + 2 t + 19 t^2 - 2 t^4)}{2^73^2 (1+t)\left(1-t^2\right)^2}+\frac{(p-1)N(t)}{2^43^2 \left(1 + t^2\right)^2 \left(1 - t^6\right)^2}-1\right.\\
-&\left. \left(\frac{(p+1)C_{-3,1}(t)}{2^33^2 (1-t)^2\left(1 + t^2 + t^4\right)^2}+\frac{ C_{-3,2}(t)}{2^23^2 (1 - t) (1 - t^2 + t^4) (1 - t^6)}\right) \left(\left(\frac{-3}{p}\right)-1\right)\right.\\
-&\left. \left(\frac{(p+1)C_{-1,1}(t)}{2^63 (1 - t) (1 + t^2) (1 - t^4)}+\frac{C_{-1,2}(t)}{2^53 (1 - t) (1 - t^4) (1 - t^2 + t^4)}\right) \left(\left(\frac{-1}{p}\right)-1\right)\right.\\
+&\left.\frac{t (-2 - 2 t + t^3)}{2^33 (1+t) (1 - t^2 + t^4)}\left(\left(\frac{-3}{p}\right)-1\right)\left(\left(\frac{-1}{p}\right)-1\right)+\frac{b^2 h^2}{2^4(1+t)}-\frac{b h}{2^2 (1-t^2)}
\right]t^3\\
+&\frac{t^3}{2^2(1 + t) (1+ t^2)}\begin{cases}
1&p\equiv 7\pmod{8},\\
0&\text{otherwise},
\end{cases}+\frac{t^3 (1 + t)}{5(1 + t + t^2 + t^3 + t^4)}\begin{cases}
1&p=5,\\
0&p\equiv 1\ \pmod{5},\\
1&p\equiv 2, 3 \hspace{-0.1in}\pmod{5},\\
2&p\equiv 4\ \pmod{5},
\end{cases}
\end{align*}
where
\begin{align*}
N(t)=&34 - 6 t + 133 t^2 - 35 t^3 + 264 t^4 - 88 t^5 + 344 t^6 - 120 t^7 + 
 342 t^8\\ -& 58 t^9 + 224 t^{10} + 86 t^{12} + 14 t^{13} + 13 t^{14} + 5 t^{15},\\
C_{-3,1}(t)=&-2 + 2 t - 6 t^2 + 6 t^3 - 5 t^4 + 3 t^5 - 5 t^6 + 3 t^7 - 3 t^8 + t^9,\\
C_{-3,2}(t)=&14 - 8 t - 10 t^2 + 6 t^3 - 5 t^4 - 2 t^5 + 20 t^6 - 6 t^7 - 16 t^8 + 
4 t^9 + 7 t^{10},\\
C_{-1,1}(t)=&-3 - 4 t + 4 t^2 - 8 t^3 + t^4 - 4 t^5 + 2 t^6,\\
C_{-1,2}(t)=&12 + 10 t - 43 t^2 + 40 t^4 - 36 t^6 + 10 t^7 + 13 t^8.
\end{align*}
\end{theorem}
\begin{theorem}\label{Theorem of IVa1}
	Let $p\geq5$ be a prime and $k\ge 3$. Then	
	\begin{align*}
	&s_k(p,{\rm IVa})\\
	=&\frac{(p-1)\left(p^3-1\right)}{2^73^35}(k-2) (k-1) (2 k-3)+\frac{7(p-1)^2(-1)^k}{2^73^2}(k-2) (k-1)+\delta_{k, 3}\\ 
	+&\frac{(p-1) \left(16 \big(\frac{-3}{p}\big)+9 \big(\frac{-1}{p}\big)-25\right)}{2^73^3}(2k-3)+\frac{\left(\big(\frac{-3}{p}\big)-1\right) \left(9 \big(\frac{-1}{p}\big)+p-10\right)}{2^33^3}\hat{c}_3(k)\\
	+&\frac{\left(16 \big(\frac{-3}{p}\big)-p-15\right) \left(9 \big(\frac{-1}{p}\big)-25\right)-16 (p+31)}{2^73^3}(-1)^k
	+\frac{\big(\frac{-1}{p}\big)-1}{2^53} (p-1) f_4(k)\\
	+&\frac{\big(\frac{-3}{p}\big)-1}{2^23^3} (p-1) f_6(k)-\frac{4\left(\big(\frac{-3}{p}\big)-2\right)}{3^3} c_6(k)+\frac{2\left(\big(\frac{-3}{p}\big)+1\right)}{3^3} \hat{c}_6(k)-\frac{4\left(\big(\frac{-3}{p}\big)+1\right)}{3^3} c^{\prime}_6(k)\\
	+&\frac{{\left(\big(\frac{-3}{p}\big)-1\right)\left(\big(\frac{-1}{p}\big)-1\right)}}{2^23} c_{12}(k)
	+\frac{c_4(k)}{2}\begin{cases}
	1&p\equiv 7\pmod{8},\\
	0&\text{otherwise},
	\end{cases}+\frac{c_5(k)}{5}\begin{cases}
	1&p=5,\\
	0&p\equiv 1\ \pmod{5},\\
	2&p\equiv 2, 3\hspace{-0.1in}\pmod{5},\\
	4&p\equiv 4\ \pmod{5}.
	\end{cases}
	\end{align*}
	
\vspace*{-0.2in}	
\noindent Equivalently,
\begin{align*}
&\sum_{k\geq 3}s_k(p,{\rm IVa})t^k\\
=&\left[\frac{(p-1)\left(p^3-1\right)(t+1)}{2^63^25 (1-t)^4}-\frac{7(p-1)^2}{2^63^2\left(1+t\right)^3}+1\right.\\
+&\left. \left(\frac{(p-1)(1+t) (3 - 5 t + 10 t^2 - 13 t^3 + 10 t^4 - 5 t^5 + 3 t^6)}{2^33^2(1 - t)^2 (1 + t^2 + t^4)^2}+\frac{2}{3^2(1+t^3)}\right)  \left(\left(\frac{-3}{p}\right)-1\right)\right.\\
+&\left.\frac{(p-1)\left(3 - 2 t^2 + 3 t^4\right)\left(\big(\frac{-1}{p}\big)-1\right)}{2^53(1 - t) (1 + t^2) (1 - t^4)}-\frac{(3 + 6 t + 7 t^2 + 6 t^3 + 3 t^4)\left(\big(\frac{-3}{p}\right)-1\big)\left(\big(\frac{-1}{p}\big)-1\right)}{2^33(1 + t) (1 + t + t^2) (1 - t^2+t^4)}\right]t^3\\
-&\frac{t^3}{2(1 + t) (1+ t^2)}\begin{cases}
1&p\equiv 7\pmod{8},\\
0&\text{otherwise},
\end{cases}-\frac{t^3(1 + t)}{5(1 + t + t^2 + t^3 + t^4)}\begin{cases}
1&p=5,\\
0&p\equiv 1\ \pmod{5},\\
2&p\equiv 2, 3 \hspace{-0.1in}\pmod{5},\\
4&p\equiv 4\ \pmod{5}.
\end{cases}
\end{align*}
\end{theorem}

\begin{theorem}\label{Theorem of Va1}
	Let $p\geq5$ be a prime and $k\ge 3$. Then	
	\begin{align*}
	&s_k(p,{\rm Va})\\
	=&\frac{p(p-1)^2}{2^83^35}(k-2) (k-1) (2 k-3)+\frac{1-\big(\frac{-1}{p}\big)}{2^4}-\frac{bh}{2^3}+\frac{b^2h^2-2bh}{2^4}(-1)^k\\
	-&\frac{7(p-1)^2(-1)^k}{2^83^2}(k-2) (k-1)+\frac{p-1}{2^43}(2k-3)(-1)^k-\frac{1-\big(\frac{-3}{p}\big)}{2^23}\hat{c}_3(k)(-1)^k\\ 
	-&\frac{(p-1) \left(-32 \big(\frac{-3}{p}\big)-27 \big(\frac{-1}{p}\big)+12 p-97\right)}{2^83^3}(2k-3)-\frac{(p-4) \left(\big(\frac{-3}{p}\big)-1\right)}{2^23^3}\hat{c}_3(k)\\
	+&\frac{\left(32 \big(\frac{-3}{p}\big)-5 p-3\right) \left(-9 \big(\frac{-1}{p}\big)+1\right)-40 (p+7)}{2^83^3}(-1)^k-\frac{(p-1) \left(\big(\frac{-1}{p}\big)-1\right)}{2^63} f_4(k)\\
	-&\frac{(p-1) \left(\big(\frac{-3}{p}\big)-1\right)}{2^33^3} f_6(k)
	+\frac{2\left(2 \big(\frac{-3}{p}\big)-1\right)}{3^3} c_6(k)+\frac{\big(\frac{-3}{p}\big)+1}{3^3} \hat{c}_6(k)-\frac{2\left(\big(\frac{-3}{p}\big)+1\right)}{3^3} c^{\prime}_6(k)\\
	-&\frac{\left(\big(\frac{-3}{p}\big)-1\right)\left(\big(\frac{-1}{p}\big)-1\right)}{2^33} c_{12}(k)-\frac{c_4(k)}{2^2}\begin{cases}
	1&p\equiv 7\pmod{8},\\
	0&\text{otherwise},
	\end{cases}+\frac{c_5(k)}{5}\begin{cases}
	1&p\equiv 2, 3\hspace{-0.1in}\pmod{5},\\
	-2&p\equiv 4\ \pmod{5},\\
	0&\text{otherwise}.
	\end{cases}
	\end{align*}

\vspace*{-0.1in}		
\noindent Equivalently,
\begin{align*}
&\sum_{k\geq 3}s_k(p,{\rm Va})t^k\\
=&\left[\frac{p\left(p-1\right)^2(1+t)}{2^73^25 (1-t)^4}+\frac{(p-1)^2(1 - 30 t - 5 t^2 + 2 t^4)}{2^73^2(1 - t)^2 (1 + t)^3}+\frac{(p-1)(5-t^2)t}{2^33\left(1 - t^2\right)^2}-\frac{b^2 h^2}{2^4 (1+t)}-\frac{b h t}{2^2 \left(1-t^2\right)}\right.\\
+&\left.\left(\frac{(p-1)(2 - 2 t + 9 t^2 - 7 t^3 + 7 t^4 - 5 t^5 + 3 t^6 - t^7)t^2}{2^33^2 (1-t)^2 \left(1 + t^2 + t^4\right)^2}+\frac{-2 + t^2 + 3 t^3 + 2 t^4}{2\cdot 3^2(1 + t) (1 + t^2 + t^4)}\right)\left(\left(\frac{-3}{p}\right)-1\right)\right.\\
-&\left.  \left(\frac{(p-1)(1 - 4 t - 4 t^2 - 8 t^3 + 5 t^4 - 4 t^5 + 2 t^6)}{2^63 (1 - t) (1 + t^2) (1 - t^4)}+\frac{1+3t}{2^5 (1 - t^2)}\right)\left(\left(\frac{-1}{p}\right)-1\right)\right.\\
+&\left.\frac{2 + 2 t + t^4}{2^33 (1+t) \left(1-t^2+t^4\right)}\left(\left(\frac{-3}{p}\right)-1\right)\left(\left(\frac{-1}{p}\right)-1\right)
\right]t^3\\
+&\frac{t^3}{2^2(1 + t) (1+ t^2)}\begin{cases}
1&p\equiv 7\pmod{8},\\
0&\text{otherwise},
\end{cases}-\frac{t^3(1 + t)}{5(1 + t + t^2 + t^3 + t^4)}\begin{cases}
1&p\equiv 2, 3\hspace{-0.1in}\pmod{5},\\
-2&p\equiv 4\ \pmod{5},\\
0&\text{otherwise}.
\end{cases}
\end{align*}
\end{theorem}

\begin{proof}[Proof of Theorems~\ref{Theorem of IIa1}-\ref{Theorem of Va1}]
By \eqref{dimskreleq} and \cite[Table 3]{Schmidt200501}, we obtain
\begin{equation}
\label{general formula for dim of spaces}
	\begin{split}
	&\dim_{\C} S_k({\rm K}(p))=2s_k(p,{\rm I})+s_k(p,{\rm IIa})+s_k(p,{\rm IIb})+s_k(p,{\rm Vb})+s_k(p,{\rm VIc}),\\
	&\dim_{\C} S_k(\Gamma_0(p))=4s_k(p,{\rm I})+s_k(p,{\rm IIa})+3s_k(p,{\rm IIb})+2s_k(p,{\rm IIIa})+s_k(p,{\rm Vb})\\
	&\hspace{52ex}+s_k(p,{\rm VIa})+s_k(p,{\rm VIb}),\\
	&\dim_{\C} S_k(\Kl(p))=4s_k(p,{\rm I})+2s_k(p,{\rm IIa})+2s_k(p,{\rm IIb})+s_k(p,{\rm IIIa})+s_k(p,{\rm Va})\\
	&\hspace{40.5ex}+s_k(p,{\rm Vb})+s_k(p,{\rm VIa})+s_k(p,{\rm VIc}),\\
	&\dim_{\C} S_k(\B(p))=8s_k(p,{\rm I})+4s_k(p,{\rm IIa})+4s_k(p,{\rm IIb})+4s_k(p,{\rm IIIa})+s_k(p,{\rm IVa})\\
	&\hspace{13ex}+2s_k(p,{\rm Va})+2s_k(p,{\rm Vb})+3s_k(p,{\rm VIa})+s_k(p,{\rm VIb})+s_k(p,{\rm VIc}).
	\end{split}
\end{equation}
Let us replace $s_k(p,{\rm VIb})$ by $s_k^{\mathbf{(G)}}(p,{\rm VIb})+s_k^{\mathbf{(Y)}}(p,{\rm VIb})+s_k^{\mathbf{(P)}}(p,{\rm VIb})$. Observing \eqref{IIIa/VI. VIa, VIb I} and \eqref{IIIa/VI. VIa, VIb II}, we get the system of equations
\begin{align}
&s_k(p,{\rm IIa})=\dim_{\C} S_k({\rm K}(p))-2s_k(p,{\rm I})-s_k(p,{\rm IIb})-s_k(p,{\rm Vb})-s_k(p,{\rm VIc}),\\
&s_k(p,{\rm IIIa+VIa/b})=\frac 12 \dim_{\C} S_k(\Gamma_0(p))-\frac 12 \dim_{\C} S_k({\rm K}(p))-s_k(p,{\rm I})-s_k(p,{\rm IIb})\nonumber\\
&\hspace{34ex}-\frac 12 s_k^{\rm \mathbf{(P)}}(p,{\rm VIb})-\frac 12 s_k^{\rm \mathbf{(Y)}}(p,{\rm VIb})+ \frac 12 s_k(p,{\rm VIc}),\\
&s_k(p,{\rm Va})=\dim_{\C} S_k(\Kl(p))-\frac 12 \dim_{\C} S_k(\Gamma_0(p))-\frac 32\dim_{\C} S_k({\rm K}(p))+s_k(p,{\rm I})\nonumber\\
&\hspace{5ex}+s_k(p,{\rm IIb})+s_k(p,{\rm Vb})+\frac 12 s_k^{\rm \mathbf{(P)}}(p,{\rm VIb})+\frac 12 s_k^{\rm \mathbf{(Y)}}(p,{\rm VIb})+\frac 12 s_k(p,{\rm VIc}),\\
&s_k(p,{\rm IVa})=\dim_{\C} S_k(\B(p))+\dim_{\C} S_k({\rm K}(p))-\dim_{\C} S_k(\Gamma_0(p))-2\dim_{\C} S_k(\Kl(p))\nonumber\\
&\hspace{58ex}+2s_k(p,{\rm I})+2s_k(p,{\rm IIb}).
\end{align}
We derive the explicit formulas of $s_k(p,\Omega)$ for $\Omega\in\{{\rm IIa, IIIa+VIa/b, IVa, Va}\}$ using Theorems~\ref{Theorem of I and IIb}, \ref{Theorem of Saito-Kurokawa type},~\ref{relations of Yoshida type} and the global dimension formulas of $S_k({\rm K}(p))$, $S_k(\Gamma_0(p))$, $S_k(\Kl(p))$ and $S_k(\B(p))$. The generating series $\sum_{k\ge 3} s_k(p,\Omega) t^k$ follow in a straightforward way. 
\end{proof}
Note that the quantities $s_k(p,{\rm IVa})$ and $s_k(p,{\rm Va})$ are the same as the quantities $n(D_{k-1,k-2}^{\rm Hol},{\rm St},p)$ and $n(D_{k-1,k-2}^{\rm Hol},{\rm Va},p)$ in \cite{Wakatsuki2013}, respectively.

\subsection*{For \texorpdfstring{$p=2,3$}{}}

\noindent Here we give rational expression for the generating function of $s_k(p,\Omega)$ for $p=2,3$. We compute them separately as follows because the formulas in Lemma~\ref{theorem for +,- new space} are different for these two primes.
\begin{align*}
&\sum\limits_{\substack{k \geq 2}}s_k(2,{\rm Vb})t^k=\frac{t^8}{(1-t^4)(1-t^6)}\qquad\qquad\
\sum\limits_{\substack{k\ge 2 }}s_k^{\rm \textbf{(P)}}(2,{\rm VIb})t^k=\frac{t^6+t^8-t^{12}}{(1-t^4)(1-t^6)}\\
&\sum\limits_{\substack{k \geq 2}}s_k(3,{\rm Vb})t^k=\frac{t^6}{(1-t^2)(1-t^6)}\qquad\qquad\
\sum\limits_{\substack{k\ge 2}}s_k^{\rm \textbf{(P)}}(3,{\rm VIb})t^k=\frac{t^4+t^8-t^{10}}{(1-t^2)(1-t^6)}\\
&\sum\limits_{\substack{k \ge 2}}s_k(2,{\rm VIc})t^k=\frac{t^{11}}{(1-t^4)(1-t^6)}\qquad\qquad
\sum\limits_{k\geq 2}s_k^{\rm \textbf{(Y)}}(2,{\rm VIb})t^{k}=0\\
&
\sum\limits_{\substack{k \ge 2}}s_k(3,{\rm VIc}) t^k=\frac{t^{9}}{(1-t^2)(1-t^6)}\qquad\qquad
\sum\limits_{k\geq 2}s_k^{\rm \textbf{(Y)}}(3,{\rm VIb})t^{k}=0\\
&\sum\limits_{\substack{k \geq 3}}s_k(2,{\rm IIa})t^k=\frac{t^{19} \left(-t^8-t^6+t^4+t^2+1\right)}{\left(1-t^4\right)^2(1-t^6)\left(1-t^{10}\right)}+\frac{t^{16} \left(t^8-t^6-t^4+t^2+1\right)}{\left(1-t^4\right)^2(1-t^6)\left(1-t^{10}\right)}\\
&\sum\limits_{\substack{k \geq 3}}s_k(3,{\rm IIa})t^k=\frac{t^{15} \left(-t^{14}-t^{12}-t^{10}+2 t^8+2t^6+t^4+t^2+1\right)}{\left(1-t^4\right)\left(1-t^6\right)^2 \left(1-t^{10}\right)}\\
&\hspace{15ex}+\frac{t^{12} \left(t^{14}-t^{12}-t^{10}+2
	t^6+t^4+t^2+1\right)}{\left(1-t^4\right)\left(1-t^6\right)^2 \left(1-t^{10}\right)}\\
&\sum\limits_{\substack{k \geq 3}}s_k(2,{\rm IIIa+VIa/b})t^k =\frac{t^{25}
	\left(-t^{10}+t^8+t^6+t^4+t^2+1\right)}{\left(1-t^4\right)\left(1-t^6\right)\left(1-t^{10}\right)\left(1-t^{12}\right)}\\
&\hspace{27ex}+\frac{t^{12} \left(t^{22}-t^{18}-t^{16}-t^{14}-2 t^{12}+t^8+2 t^6+2
	t^4+2 t^2+1\right)}{\left(1-t^4\right)\left(1-t^6\right)\left(1-t^{10}\right)\left(1-t^{12}\right)}\\	
&\sum\limits_{\substack{k \geq 3}}s_k(3,{\rm IIIa+VIa/b}) t^k=\frac{t^{17} \left(-t^{18}+t^{16}+2 t^{14}+2 t^{12}+2 t^{10}+3 t^8+2t^6+t^4+t^2+1\right)}{\left(1-t^4\right)\left(1-t^6\right)\left(1-t^{10}\right)\left(1-t^{12}\right)}\\
&\hspace{15ex}+\frac{t^8 \left(t^{26}-t^{22}-2 t^{20}-2 t^{18}-2 t^{16}+t^{12}+4
	t^{10}+5 t^8+4 t^6+3 t^4+2 t^2+1\right)}{\left(1-t^4\right)\left(1-t^6\right)\left(1-t^{10}\right)
	\left(1-t^{12}\right)}\\
&\sum\limits_{\substack{k \geq 3}}s_k(2,{\rm IVa})t^k=\frac{t^{13} \left(t^{22}-t^{18}-t^{16}+t^{12}+2 t^8+2
	t^6+t^4+t^2+1\right)}{\left(1-t^4\right)\left(1-t^6\right)\left(1-t^{10}\right)\left(1-t^{12}\right)}\\
&\hspace{16ex}+\frac{t^{10}(-t^{22}+t^{18}+t^{16}+t^{12}+2
	t^{10}+t^{4}+t^{2}+1)}{\left(1-t^4\right)\left(1-t^6\right)\left(1-t^{10}\right)\left(1-t^{12}\right)}\\
&\sum\limits_{\substack{k \geq 3}}s_k(3,{\rm IVa}) t^k=\frac{t^9 \left(t^{26}-t^{22}+2 t^{18}+5 t^{16}+5 t^{14}+9 t^{12}+9
	t^{10}+8 t^8+6 t^6+5 t^4+2 t^2+1\right)}{\left(1-t^4\right)\left(1-t^6\right)\left(1-t^{10}\right)
	\left(1-t^{12}\right)}\\
&\hspace{7ex}+\frac{t^6 \left(-t^{26}+t^{22}+2 t^{20}+2 t^{18}+5 t^{16}+7 t^{14}+7
	t^{12}+7 t^{10}+8 t^8+6 t^6+5 t^4+2
	t^2+1\right)}{\left(1-t^4\right)\left(1-t^6\right)\left(1-t^{10}\right)
	\left(1-t^{12}\right)}\\
&\sum\limits_{\substack{k \geq 3}}s_k(2,{\rm Va})
t^k=\frac{t^{15}(-t^{12}+t^{2}+1)+t^{30}}{\left(1-t^4\right)\left(1-t^6\right)\left(1-t^{10}\right)\left(1-t^{12}\right)}\\
&\sum\limits_{\substack{k \geq 3}}s_k(3,{\rm Va}) t^k=\frac{t^{11} \left(-t^{18}-t^{16}+2 t^8+2 t^6+2t^4+t^2+1\right)+t^{16}
	(1 + t^4 + t^6) (1 + t^8)}{\left(1-t^4\right)\left(1-t^6\right)\left(1-t^{10}\right)
	\left(1-t^{12}\right)}
\end{align*}
Note that the series $\sum_{k\ge 3} k^n t^k$ has a pole of order $n+1$ at $t=1$. The pole of order $4$ at $t=1$ in the rational expression of $\sum_{k\geq 3}s_k(p,\Omega)t^k$ for $\Omega \in \left\{\rm I, IIa, IIIa+VIa/b, IVa, Va\right\}$ is coming from the term $(k-2) (k-1) (2 k-3)$ in $s_k(p,\Omega)$. In fact, we have the following results.
\begin{corollary} 
	\label{Corollary for sk and Plancherel measure}
	Let $p\ge 2$ be a prime and $k\ge 3$. Then, for $\Omega \in \left\{\rm I, IIa, IIIa+VIa/b, IVa, Va\right\}$,
	  \begin{equation}\label{relationship of sk and Plancherel measure}
	\begin{split}
	s_k(p,\Omega)&=a_{\Omega}\cdot \frac{(k-2) (k-1) \big(\frac k3-\frac 12\big)}{2^63^25}+b_{\Omega}\cdot \dfrac{7(-1)^k}{2^{7}3^{2}}(k-2) (k-1)+O(k),
	\end{split}
	\end{equation}
where $a_{\Omega}$ and $b_{\Omega}$ are given as follows.
\begin{equation}
\label{Coefficients of Plancherel measure term and $(k-2)(k-1)$ term}
\renewcommand{\arraystretch}{1.4}
\renewcommand{\arraycolsep}{.6cm}
\begin{array}{cccccc} 
\toprule
\Omega&{\rm I}&	{\rm IIa}&	{\rm IIIa+VIa/b}&	{\rm IVa}&	{\rm Va}\\
\toprule
a_{\Omega}&1&p^2-1&\frac{(p-1)(p^2+p+2)}{2}&(p-1)(p^3-1)&\frac{p(p-1)^2}{2}\\
\toprule
b_{\Omega}&1&0&\frac{(p-1)(p+3)}{2}&(p-1)^2&-\frac{(p-1)^2}{2}\\
\bottomrule
\end{array}
\end{equation}
\end{corollary}
\begin{proof}
For $p\ge 5$, the result follows from Theorems~\ref{Theorem of I and IIb}, \ref{Theorem of IIa1}-\ref{Theorem of Va1}. For $p=2,3$, we note that the second term of \eqref{relationship of sk and Plancherel measure} does not contribute a pole at $t=1$ in $\sum_{k\geq 3}s_k(p,\Omega)t^k$ and the rational expression for 
$$\sum_{k\geq 3}s_k(p,\Omega)t^k -\sum_{k\geq 3}a_{\Omega}\cdot \frac{(k-2) (k-1) \big(\frac k3-\frac 12\big)}{2^63^25}t^k$$
has a pole of order $2$  at $t=1$. Hence, we obtain \eqref{relationship of sk and Plancherel measure} for $p=2,3$ as well.
\end{proof}
In Sect.~\ref{Relationship with the Plancherel measure} we will show that $a_\Omega$ equals the total Plancherel measure of the tempered Iwahori-spherical representations of $\PGSp(4,\Q_p)$ of type $\Omega$. We do not know a similar interpretation for the quantity $b_\Omega$.

\subsection{The cases \texorpdfstring{$k=1$}{} and \texorpdfstring{$k=2$}{}}
The formulas in the theorems in the previous section hold for $k\geq3$. We now consider $k=1$ and $k=2$.
\begin{proposition}\label{k1prop}
We have $s_1(p, \Omega)=0$ for all $\Omega$ in Table~\ref{rep with Arthur type}.
\end{proposition}
\begin{proof}
By \cite[Theorem~6.1]{Ibukiyama2007}, the left hand sides of the equations \eqref{general formula for dim of spaces} are all zero. Since all the quantities on the right hand sides are non-negative numbers, this implies our assertion.
\end{proof}

We cannot determine the numbers $s_2(p,\Omega)$ in general because of a lack of global dimension formulas, but we can at least treat the cases $p=2$ and $p=3$ (see Table~\ref{historytable}). The following lemma is useful for proving that $\dim_{\C} S_2(\B(3))=0$.

\begin{lemma}\label{Lemma for I(p)}
 Let $k$ be a positive integer, and let $\Gamma$ be either $\Gamma_0(p)$ or $\Kl(p)$. Suppose that $f\in S_k(\B(p))$. If $f^2\in S_{2k}(\Gamma)$, then $f\in S_k(\Gamma)$.
\end{lemma}
\begin{proof}
Suppose $f\in S_k(\B(p))$ and $f^2\in S_{2k}(\Gamma)$. We have the slash operator $(f|_{k}\gamma)(Z):=(CZ+D)^{-k}f(\gamma Z)$ for $\gamma=\begin{bsmallmatrix}A&B\\C&D\end{bsmallmatrix}\in \Sp(4,\Z)$. Since $(f|_k\gamma)^2=f^2|_{2k}\gamma=f^2$ for $\gamma\in\Gamma$, there exists a sign $\varepsilon_\gamma\in\{\pm1\}$ such that $f|_k\gamma=\varepsilon_\gamma f$. Since $f|_k\gamma_1\gamma_2=(f|_k\gamma_1)|_k \gamma_2$, we see that $\varepsilon:\Gamma\rightarrow \{\pm 1\}$ is a homomorphism and $f\in S_k(\Gamma, \varepsilon)$.

We will show that $\varepsilon$ is trivial. Since $f\in S_k(\B(p))$, the kernel $H$ of $\varepsilon$ contains $B(p)$. Hence $H$ is a normal subgroup of $\Gamma$ with $B(p)\subset H\subset\Gamma$. This implies $H=\Gamma$, which is maybe most easily seen by applying the projection mod $p$.
\end{proof}

\begin{proposition}\label{k2prop}
We have $s_2(p, \Omega)=0$ for $p\in\{2,3\}$ and all $\Omega$ in Table~\ref{rep with Arthur type}.
\end{proposition}
\begin{proof}
We already know this is true for $\Omega={\rm I}$ and for all the Saito-Kurokawa and Yoshida types. Hence the formulas in \eqref{general formula for dim of spaces} become
\begin{equation}\label{k2propeq1}
	\begin{split}
	&\dim_{\C} S_2({\rm K}(p))=s_2(p,{\rm IIa}),\\
	&\dim_{\C} S_2(\Gamma_0(p))=s_2(p,{\rm IIa})+2s_2(p,{\rm IIIa}+{\rm VIa/b}),\\
	&\dim_{\C} S_2(\Kl(p))=2s_2(p,{\rm IIa})+s_2(p,{\rm IIIa}+{\rm VIa/b})+s_2(p,{\rm Va}),\\
	&\dim_{\C} S_2(\B(p))=4s_2(p,{\rm IIa})+4s_2(p,{\rm IIIa}+{\rm VIa/b})+s_2(p,{\rm IVa})+2s_2(p,{\rm Va}).
	\end{split}
\end{equation}
Most of the dimensions on the left, except for $S_2(\B(3))$, are known to be zero by \cite[Sect.~5.3]{Ibukiyama2018} and \cite[Sect.~1]{Ibukiyama1984}. It follows that the only potentially non-zero number is
\begin{equation}\label{k2propeq2}
 \dim_{\C} S_2(\B(3))=s_2(3,{\rm IVa}).
\end{equation}
Suppose that there exists a non-zero $f\in S_2(\B(3))$. We have $\dim_{\C} S_4(\B(3))=\dim_{\C}  S_4(\Gamma_0(3))=1$ by \cite[Sect.~2.4]{Ibukiyama2007}. In particular, $S_4(\B(3))=S_4(\Gamma_0(3))$. The function $f^2$ spans this $1$-dimensional space. Lemma~\ref{Lemma for I(p)} then implies that $f\in S_2(\Gamma_0(3))$. But $S_2(\Gamma_0(3))=0$ by \cite[Sect.~5.3]{Ibukiyama2018}, a contradiction. So, $s_2(3,{\rm IVa})=\dim_{\C} S_2(\B(3))=0$.
\end{proof}

For $\Omega=\{{\rm IIa, IIIa+VIa/b, IVa, Va}\}$, we give some numerical examples of $s_k(p,\Omega)$ for $2\le p<20$ in Appendix~\ref{AppendixB}.

\subsection{Dimensions of the spaces of newforms}\label{newformssec}
In this section we discuss some of the available notions of newforms for the spaces of Siegel modular forms and write their dimensions in terms of the quantities $s_k(p,\Omega)$.

There is no uniform definition of the spaces of newforms for Siegel modular forms for all the congruence subgroups in \eqref{classical congruence subgroups}, but there have been several attempts in the literature to define a good notion of Siegel modular newforms. Ibukiyama  defines old- and newforms
for the minimal congruence subgroup $\B(p)$ in \cite{Ibukiyama1984}, and for the 
paramodular group ${\rm K}(p)$ in \cite{Ibukiyama1985}. He provides further evidence to support these  definitions.
There is a definition of  newforms for $\Gamma_0(N)$ for any $N$ by Andrianov \cite{Andrianov1999}.
These definitions coincide with the notion of newforms in \cite{ Schmidt200501} for $\B(p)$, $\Gamma_0(p)$, and ${\rm K}(p)$.

Now there is a well established newform theory for Siegel modular forms with respect to the paramodular group ${\rm K}(p)$; see \cite{RobertsSchmidt2006, Schmidt200501}. A Siegel modular form with respect to the paramodular group is called a paramodular form. Let $S_k^{{\rm new}}({\rm K}(p))$ be the space of paramodular newforms, and let $S_k^{\mathrm{new},\mathrm{\bf(G)}}({\rm K}(p))$ be the space of paramodular newforms of {\bf(G)} type. Using the local and global newform theory, we can write $ \dim_{\C} S_k^{{\rm new}}({\rm K}(p))$  and $\dim_{\C} S_k^{\mathrm{new},\mathrm{\bf(G)}}({\rm K}(p))$ in terms of the quantities $s_k(p,\Omega)$ as follows.
\begin{proposition}\label{Proposition of newform formula of K(p)}
	Suppose $k\geq 1$ and $p$ is a prime. Then
	\begin{align*}    
	\dim_{\C} S_k^{{\rm new}}({\rm K}(p))&=s_k(p,{\rm IIa})+s_k(p,{\rm Vb})+s_k(p,{\rm VIc}),\\
	\dim_{\C} S_k^{\mathrm{new},\mathrm{\bf(G)}}({\rm K}(p))&=s_k(p,{\rm IIa}).
	\end{align*}
\end{proposition}
\begin{proof}
By \cite[Theorem~3.3.12]{Schmidt200501}, there is a paramodular newform $f \in S_k^{{\rm new}}({\rm K}(p))$ that corresponds to a cuspidal automorphic representation $\pi \in S_k(p,\Omega)$ such that $\Omega$ is one of the types ${\rm IIa}$, ${\rm Vb}$ and ${\rm VIc}$. Out of these three types, only a representation of type ${\rm IIa}$ can appear as a local component of a cuspidal automorphic representation of {\bf(G)} type. The assertions follow.
\end{proof}
There is a definition of the space $S_k^{{\rm new}}(\B(p))$ of newforms with respect to the Iwahori subgroup $\B(p)$ in \cite[Sect.~3.3]{Schmidt200501}. Using this definition and \cite[Theorem~3.3.2]{Schmidt200501}, it is easy to see that if $\pi \in S_k(p,\Omega)$ is the cuspidal automorphic representation associated to $f \in S_k^{{\rm new}}(\B(p))$, then $\Omega$ is of type $\rm IVa$. Hence we get the following result.
\begin{proposition}\label{Proposition of newform formulas I(p)}
	Suppose $k\geq 1$ and $p$ is a prime. Then
	\begin{align*}
	\dim_{\C} S_k^{{\rm new}}(\B(p))&=s_k(p,{\rm IVa}).
	\end{align*}
\end{proposition}
Similarly, using the definition of the space $S_k^{{\rm new}}(\Gamma_0(p))$ of newforms with respect to the Siegel congruence subgroup $\Gamma_0(p)$ and \cite[Theorem~3.3.9]{Schmidt200501}, we get the following result.
\begin{proposition}\label{Proposition of newform formulas Si(p)}
	Suppose $k\geq 1$ and $p$ is a prime. Then
	\begin{align*}
	\dim_{\C} S_k^{{\rm new}}(\Gamma_0(p))&=s_k(p,{\rm IIa})+2s_k(p,{\rm IIIa+VIa/b})\\
	&+s_k(p,{\rm Vb})+s_k^{\rm \mathbf{(P)}}(p,{\rm VIb})+s_k^{\rm \mathbf{(Y)}}(p,{\rm VIb}).
	\end{align*}
\end{proposition}

There is no definition of the newforms with respect to the Klingen congruence subgroup given in \cite{Schmidt200501}. But one can define the space $S_k^{{\rm new}}(\Kl(p))$ in a similar manner as the space $S_k^{{\rm new}}(\Gamma_0(p))$ is defined in 
	\cite{Schmidt200501}. In that case one would have the following.
\begin{remark}
	For $k\geq 1$ and  any prime $p$, 
	$$\dim_{\C} S_k^{{\rm new}}(\Kl(p))=s_k(p,{\rm IIIa+VIa/b})+s_k(p,{\rm Va}).$$
\end{remark}

\section{Plancherel measures}
In this section we relate the global quantity $s_k(p,\Omega)$ to the total Plancherel measure of the tempered Iwahori-spherical representations of $\PGSp(4,\Q_p)$ of type $\Omega$, which is a local quantity. In Sect.~\ref{Relationship with the Plancherel measure}, we compute the local Plancherel measures of the tempered Iwahori-spherical representations of $\PGSp(4,\Q_p)$. In Sect.~\ref{Section on More general limit multiplicity formulas}, we derive that the local representation types at $v=p$ as $\pi\cong\otimes_{v} \pi_v $ varies in $S_k(p,\Omega)$ are equidistributed with respect to their Plancherel measure. We establish a similar result for the vector-valued case using the automorphic Plancherel density theorem from \cite{Shin2012}.

\subsection{Calculation of local Plancherel measures}\label{Relationship with the Plancherel measure}
Let $F$ be a non-archimedean local field of characteristic zero with residual characteristic $p$. Let $\OF$ be the ring of integers of $F$ with the maximal ideal $\p$ and let $q$ be the order of the residue field of $F$. In this section,
let $G=\PGSp(4,F)$. Let $K$ be the image of $\GSp(4,\OF)$ in $G$. We fix a Haar measure $\mu$ on $G$ for which $K$ has volume $1$. Let $\hat{G}$ be the tempered unitary dual of $G$. There is a unique Borel measure $\hat{\mu}$ on $\hat{G}$, called the \emph{Plancherel measure} with respect to $\mu$, characterized by 
\begin{equation}
\label{Plancherel measure}
f(1)=\int_{\hat{G}} {\rm Tr}(\pi(f))\,d\hat{\mu}(\pi)
\end{equation}
for all locally constant, compactly supported functions $f\colon G\rightarrow\C$ and $\pi\in \hat{G}$. It is well known that the Plancherel measure is supported on the tempered dual $\hat{G}^{\rm temp}$, and that a representation $\pi$ is square-integrable if and only if the point $\pi$ in $\hat{G}$ has positive Plancherel measure.

For $\Omega \in \left\{{\rm I}, {\rm II}, {\rm III}, {\rm IV}, {\rm V}, {\rm VI}\right\}$, let $\hat{G}_{\Omega}$ be the part of the unitary dual that consists of all Iwahori-spherical representations of type $\Omega$. Let $m_{\Omega}$ be the total Plancherel measure of the Iwahori spherical representations of type $\Omega$, defined by 
\begin{equation}
\label{Plancherel measure Iwahori-spherical}
 m_{\Omega}=\int_{\hat{G}_{\Omega}}\,d\hat{\mu}(\pi).
\end{equation}
For example, $m_{\rm I}$ denotes the total Plancherel measure of the unramified dual. The quantity $m_{\rm II}$ is the total Plancherel measure of all representations of the form $\chi {\rm St}_{\GL(2)} \rtimes \sigma$, where $\chi,\sigma$ are unramified, unitary characters of $F^{\times}$ satisfying $\chi^2\sigma^2= 1$. Note that these are of type IIa, and that the non-tempered representations of type IIb do not contribute to $m_{\rm II}$. Similarly, only representation type IIIa contributes to $m_{\rm III}$, only representation type IVa contributes to $m_{\rm IV}$, and only representation type Va contribute to $m_{\rm V}$. Also,
\begin{equation}
\label{Plancherel measure for VI}
m_{\rm VI}=0,
\end{equation}
since the four tempered representations of type $\rm VIa/b$ are not square-integrable.

Now we consider the following open-compact subgroups of $\GSp(4,F)$ defined in \eqref{congruence subgroups for GSp4F}: $K=\GSp(4,\OF)$, ${\rm K}(\p)$, ${\rm Kl}(\p)$, ${\rm Si}(\p)$, and $I$. We use the same symbol for the images of these groups in $G =\PGSp(4,F)$.

\begin{lemma}
	\label{dimesion formulas}
We have
	\begin{align*}
	{\rm Vol}_{\mu}({\rm Kl}(\p))&=\dfrac{1}{(1+q)(1+q^2)}, &\quad {\rm Vol}_{\mu}({\rm K}(\p))&=\dfrac{1}{1+q^2}, \\
	{\rm Vol}_{\mu}({\rm Si}(\p))&=\dfrac{1}{(1+q)(1+q^2)}, &\quad {\rm Vol}_{\mu}(I)&=\dfrac{1}{(1+q)^2(1+q^2)}. 
	\end{align*}
\end{lemma}
\begin{proof}
For ${\rm Kl}(\p)$ and ${\rm K}(\p)$, see \cite[Lemma~3.3.3]{RobertsSchmidt2007}. One can prove the cases for ${\rm Si}(\p)$ and $I$ in a similar manner. 
\end{proof}

Let $H$ be one of these five subgroups of $G$. Let $(\pi, V)$ be an Iwahori-spherical representation of $G$, and let $V^H$ be the space of $H$-fixed vectors. Let $f$ be the characteristic function of $H$. Then
\begin{equation}
{\rm Tr}(\pi(f))={\rm Vol}_{\mu}(H)\dim(V^H).
\end{equation}
Here, ${\rm Vol}_{\mu}(H)$ is the volume of $H$ with respect to $\mu$.
The quantity $\dim (V^H)$ is the same across all tempered representations in $\hat{G}_{\Omega}$, if, for type ${\rm VI}$, we view the $L$-packet ${\rm VIa/b}$ as one representation. We denote this common dimension by $d_{H,\Omega}$. Then, by \eqref{Plancherel measure} and \eqref{Plancherel measure Iwahori-spherical}, we get 
\begin{equation}
\label{relation between volume and Plancherel measure}
\begin{split}
1&=\sum\limits_{\Omega \in \left\{{\rm I}, {\rm II}, {\rm III}, {\rm IV}, {\rm V}, {\rm VI}\right\}}\int_{\hat{G}_{\Omega}} {\rm Tr}(\pi(f))\,d\hat{\mu}(\pi)\\
&={\rm Vol}_{\mu}(H)\sum\limits_{\Omega \in \left\{{\rm I}, {\rm II}, {\rm III}, {\rm IV}, {\rm V}, {\rm VI}\right\}}\,d_{H,\Omega}\cdot m_{\Omega}.
\end{split}
\end{equation}
We can now compute all the $m_\Omega$.
\begin{theorem}\label{Plancherel measure for Iwahori-spherical representations}
 The total Plancherel measure $m_{\Omega}$ of the tempered Iwahori-spherical representations of $\PGSp(4,F)$ of type $\Omega$ is given as follows.
	\begin{equation}
	\label{Table of Plancherel measure of Iwahori-spherical representations}
	\renewcommand{\arraystretch}{0.8}
	\renewcommand{\arraycolsep}{.3cm}
	\begin{array}{ccccccc} 
	\toprule
	\Omega&{\rm I}&	{\rm IIa}&	{\rm IIIa}&	{\rm IVa}&	{\rm Va}&	{\rm VIa/b}\\
		\toprule
	m_{\Omega}&1&q^2-1&\frac{(q-1)(q^2+q+2)}{2}&(q-1)(q^3-1)&\frac{q(q-1)^2}{2}&0\\
	\bottomrule
	\end{array}
	\end{equation}
\end{theorem}
\begin{proof}
	Running through all $H\in \left\{K,{\rm K}(\p), {\rm Kl}(\p),{\rm Si}(\p), I\right\}$, we get the following system of equations using \cite[Table 3]{Schmidt200501}, Lemma~\ref{dimesion formulas}, and \eqref{relation between volume and Plancherel measure}:
	\begin{align*}
	\label{equations with local dim}
	K&: &\quad 1&=m_{\rm I} \nonumber\\
	{\rm K}(\p)&: &\quad (1+q^2)&=2m_{\rm I}+m_{\rm II}\nonumber\\
	{\rm Kl}(\p)&: &\quad (1+q)(1+q^2)&=4m_{\rm I}+2m_{\rm II}+m_{\rm III}+m_{\rm V}+m_{\rm VI}\\
	{\rm Si}(\p)&: &\quad (1+q)(1+q^2)&=4m_{\rm I}+m_{\rm II}+2m_{\rm III}+2m_{\rm VI}\nonumber\\
	I&: &\quad (1+q)^2(1+q^2)&=8m_{\rm I}+4m_{\rm II}+4m_{\rm III}+m_{\rm IV}+2m_{\rm V}+4m_{\rm VI}.\nonumber
	\end{align*}
	Then, using \eqref{Plancherel measure for VI}, the discussion after \eqref{Plancherel measure Iwahori-spherical}, and the above system of equations, we get \eqref{Table of Plancherel measure of Iwahori-spherical representations}.
\end{proof}
Let us combine, as before, the types III and VI. Hence we define $m_{\Omega}=m_{\rm III}+m_{\rm VI}$ ($=m_{\rm III}$) for $\Omega=\rm III+VI$. Assume that $F=\Q_p$. Then, comparing \eqref{Coefficients of Plancherel measure term and $(k-2)(k-1)$ term} and \eqref{Table of Plancherel measure of Iwahori-spherical representations}, we see that the coefficient $a_\Omega$ appearing in \eqref{relationship of sk and Plancherel measure} equals the Plancherel mass $m_{\Omega}$:
\begin{equation}\label{amequaleq}
 a_\Omega=m_\Omega\qquad\text{for }\Omega\in\{\text{I},\text{II},\text{III+VI},\text{IV},\text{V}\}.
\end{equation}
Here, we allowed ourselves to write $a_{\text{II}}$ for $a_{\text{IIa}}$ etc.

\subsection{More general limit multiplicity formulas} \label{Section on More general limit multiplicity formulas}
Let $k\geq3$ and $j\geq0$ be integers. In the following, let $\xi_{k,j}$ be the irreducible, finite-dimensional representation of $\Sp(4,\C)$ with highest weight $(k+j-3,k-3)$, where the weight is defined as in \cite{Schmidt2017}. By \cite{CaglieroTirao2004}, it is known that
\begin{equation}\label{findimrepformulaeq}
 \dim \xi_{k,j}=\frac{(j+1)(k-2)(k+j-1)(2k+j-3)}{6}.	
\end{equation}
The infinitesimal character of $\xi_{k,j}$ is $(k+j-1,k-2)$. Let $\Pi_{\rm disc}(k,j)$ be the $L$-packet consisting of the discrete series representations of $\PGSp(4,\R)$ with the same infinitesimal character $(k+j-1,k-2)$. It consists of a holomorphic (non-generic) discrete series representation $\mathcal{D}^{\rm hol}(k,j)$ with minimal $K$-type $(k+j,k)$ and a large (generic) discrete series representation $\mathcal{D}^{\rm gen}(k,j)$ with minimal $K$-type $(k+j,2-k)$. The representation $\mathcal{D}^{\rm hol}(k,j)$ is the archimedean component of the automorphic representations underlying vector-valued Siegel modular forms of weight $(k,j)$.

Note that, by Corollary~\ref{Corollary for sk and Plancherel measure} and equations \eqref{amequaleq} and \eqref{findimrepformulaeq},
\begin{equation}\label{limmulteq1}
 \lim_{k\rightarrow\infty}\frac{s_k(p,\Omega)}{2^{-6}3^{-2}5^{-1}\dim \xi_{k,0}}=m_{\Omega}.
\end{equation}
In this form, the result is reminiscent of the automorphic Plancherel density theorem in \cite{Shin2012} for $G=\PGSp(4)$, more precisely Corollary 4.12 of this paper. Roughly speaking, the automorphic Plancherel density theorem says that in a family of global representations for which the archimedean parameter tends to infinity, and which is unramified outside a finite set of finite places $T$, the local representations at the places in $T$ are equidistributed according to the Plancherel measure. However, there is a difference in that \cite{Shin2012} considers archimedean $L$-packets, whereas our representations are all holomorphic at infinity.

To clarify the relationship, we assume that $k\geq3$ and define the set $\tilde S_k(p,\Omega)$ as in Definition~\ref{Definition of mOmegak}, except we replace condition \textit{i)} by ``$\pi_\infty\in\Pi_{\rm disc}(k,0)$''. Let $\tilde s_k(p,\Omega)=\#\tilde S_k(p,\Omega)$. We also define Arthur type versions of these quantities, as in \eqref{SkArthurdecompeq} and \eqref{no. of cup form0}. Then, by the stability of the packets of type \textbf{(G)}, we have
\begin{equation}\label{stimes2eq}
 \tilde s_k^{\rm (\mathbf{G})}(p,\Omega)=2s_k^{\rm (\mathbf{G})}(p,\Omega).
\end{equation}
We note that $s_k(p,\Omega)$ can be replaced by $s_k^{\rm (\mathbf{G})}(p,\Omega)$ in \eqref{limmulteq1}. The reason is that, similarly as in \eqref{relations of Saito-Kurokawa type} and \eqref{general formula for Yoshida type}, the $s_k^{(*)}(p,\Omega)$ for the non-\textbf{(G)} types grow at most linearly or quadratically in $k$.
Substituting into \eqref{limmulteq1} and reverting back to all packet types with the same argument, we see that \eqref{limmulteq1} is equivalent to
\begin{equation}\label{limmulteq2}
 \lim_{k\rightarrow\infty}\frac{\tilde s_k(p,\Omega)}{2^{-5}3^{-2}5^{-1}\dim \xi_{k,0}}=m_{\Omega}.
\end{equation}
In this form the statement almost follows from Corollary 4.12 of \cite{Shin2012} (see the proof of Theorem \ref{Theorem gerenal PDT} below), except that the highest weights of $\xi_{k,0}$ lie on a wall, and hence this particular sequence of finite-dimensional representations does not satisfy the conditions in \cite[Def.~3.5]{Shin2012}. However, Lemma~4.9 of \cite{Shin2012}, where this hypothesis is used, can easily be proven directly for the $\xi_{k,0}$, using the known weight multiplicities from \cite{CaglieroTirao2004}.

\begin{remark}
 For $G=\PGSp(4)$, the signed measure $\bar{\mu}(G(\Q)\backslash G(\A_{\Q}))$ appearing in Definition~3.5 of \cite{Shin2012} is given by
  \begin{equation}\label{barmuformulaeq}
   \bar{\mu}(G(\Q)\backslash G(\A_{\Q}))=\frac{-1}{2^5 3^2 5}.
  \end{equation}
\end{remark}
This can be seen in two different ways. First, one can reduce it to a question for $\Sp(4)$ and then use  \cite[Theorems 4 and 5]{Serre1971}. Second, by considering the vector-valued case, one can reverse engineer the constant by letting $\hat U$ be the set of spherical, tempered representations at a single place in \cite[Theorem 4.11]{Shin2012} and comparing with the known formula \cite[Theorem 7.1]{Wakatsuki2012}.

Hence we see that \eqref{limmulteq1} does indeed follow from Corollary 4.12 of \cite{Shin2012}. In fact, we can use this corollary and the same arguments to expand the result \eqref{limmulteq1} so as to include more general weights and more than one ramified place. To this end we first make our Definition~\ref{Definition of mOmegak} more general.
\begin{definition}\label{Definition of SkSOmega}
 Let $k\geq3$ and $j\geq0$ be integers, and let $T$ be a finite set of finite primes. For $p\in T$, let $\Omega_p\in\{\text{\rm I},\text{\rm II},\text{\rm III+VI},\text{\rm IV},\text{\rm V}\}$, and let $\Omega_T=(\Omega_p)_{p\in T}$. We denote by $S_{k,j}(T,\Omega_T)$ the set of cuspidal automorphic representations $\pi\cong\otimes_{v \le \infty} \pi_v$ of $G(\A_{\Q})$ satisfying the following properties:
 \begin{enumerate}
  \item  $\pi_{\infty}=\mathcal{D}^{\rm hol}(k,j)$. 
  \item  For each $v\not\in T\cup \{\infty\}$, $\pi_{v}$ is unramified. 
  \item  For each $p\in T$, $\pi_{p}$ is an Iwahori-spherical representation of $G(\Q_p)$ of type $\Omega_p$.
 \end{enumerate}
 Let $s_{k,j}(T,\Omega_T)=\#S_{k,j}(T,\Omega_T)$.
\end{definition}
\begin{theorem}\label{Theorem gerenal PDT}
 Let $T$ and $\Omega_T$ be as in Definition~\ref{Definition of SkSOmega}. Let $m_{\Omega_p}$ be the total Plancherel measure of the tempered Iwahori-spherical representations of $\PGSp(4,\Q_p)$ of type $\Omega_p$, given in \eqref{Table of Plancherel measure of Iwahori-spherical representations}. Then 
 \begin{equation}\label{limmulteq3}
  \lim_{k+j\rightarrow\infty}\frac{s_{k,j}(T,\Omega_T)}{2^{-6}3^{-2}5^{-1}\dim \xi_{k,j}}=\prod\limits_{\substack{p\, \in\, T}} m_{\Omega_p},
 \end{equation}
where $\dim \xi_{k,j}$ is given in \eqref{findimrepformulaeq}.
\end{theorem}
\begin{proof}
We will use the notations of \cite[Theorem~4.11]{Shin2012}. Let $\phi^{S,\infty}$ be the characteristic function of $\prod_{p \not\in T}G(\Z_p)$. Let $\hat{f}_T$ be the characteristic function, on the unitary dual of $\prod_{p\in T}G(\Q_p)$, of those representations that are Iwahori-spherical of type $\Omega_p$ at $p\in T$. Then, unraveling the definition of the measure $\widehat{\mu}^{\rm cusp}_{\phi^{S,\infty},\xi_{k,j}}$ in \cite[Eq.~(3.5)]{Shin2012}, we see
\begin{align}
	&\widehat{\mu}^{\rm cusp}_{\phi^{S,\infty},\xi_{k,j}}(\hat{f}_T)
	=\frac{(-1)^{\dim(G(\R)/K_{\infty})/2}}{\bar{\mu}(G(\Q)\backslash G(\A_{\Q})) \dim \xi_{k, j}}\left[\tilde  s_{k,j}^{\bf(G)}(T,\Omega_T)+\tilde  s_{k,j}^{\bf(P)}(T,\Omega_T)+\tilde  s_{k,j}^{\bf(Y)}(T,\Omega_T)\right].
\end{align}
Here $K_{\infty}$ is the maximal compact subgroup of $G(\R)$ and we have $\dim(G(\R)/K_{\infty})=6$. The $\tilde s_{k,j}^{\bf(*)}(T,\Omega_T)$ are defined similarly to in the scalar-valued case above. Then, using \eqref{barmuformulaeq} and an argument similar to the one above that connected \eqref{limmulteq1} and \eqref{limmulteq2}, we obtain \eqref{limmulteq3}.
\end{proof}
\begin{remark}
Instead of Iwahori-spherical representations, one can work with any relatively quasi-compact subset of the unitary dual of $\prod_{p\in T}G(\Q_p)$. We restricted ourselves to Iwahori-spherical representations because for these we know the right hand side of \eqref{limmulteq3} explicitly.
\end{remark}

\begin{appendix}
\section{Some numerical examples}\label{AppendixB}
Here we give some numerical values of $s_k(p,\Omega)$  for $\Omega \in \lbrace \rm IIa, IIIa+VIa/b, Va, IVa\rbrace$ using Theorems~\ref{Theorem of IIa1}-\ref{Theorem of Va1}. We have used Mathematica \cite{Mathematica} to compute these numerical examples. 

	\begin{table}[H]	
	\caption{$s_k(p,{\rm IIa})$ for $3\leq k\le 35$ and $2\le p<20$.}
		\label{Table for mIIa odd}
		\renewcommand{\arraystretch}{0.8}
		\renewcommand{\arraycolsep}{.2cm}
		\[ \begin{array}{cccccccccccccccccccccccccc} 
		\toprule	
		p\backslash k&3&5&7&9&11&13&15&17&19&21&23&25&27&29&31&33&35
		\\\toprule
		2&0&0&0&0&0&0&0&0&1&1&3&2&5&5&8&9&12
		\\\toprule
		3&0&0&0&0&0&0&1&1&2&5&6&7&13&14&18&26&28
		\\
		\toprule
		5&0&0&0&0&1&2&4&7&11&16&23&30&40&51&64&79&96
		\\
		\toprule
		7&0&0&0&1&3&6&10&17&25&36&50&65&85&108&134&165&199
		\\
		\toprule
		11&0&0&0&2&7&14&26&42&63&90&124&164&213&270&336&412&499
		\\
		\toprule
		13&0&0&2&6&14&26&43&67&98&137&186&243&313&394&488&596&718
		\\
		\toprule
		17&0&0&2&9&22&42&71&112&164&231&314&412&531&670&830&1015&1224
		\\
		\toprule
		19&0&1&4&13&30&55&93&144&210&294&398&522&671&845&1046&1277&1540
		\\
		\bottomrule
		\end{array}\]
	\end{table}

	\begin{table}[H]
		\label{Table for mIIa even}
		\renewcommand{\arraystretch}{0.8}
		\renewcommand{\arraycolsep}{.23cm}
		\[ \begin{array}{cccccccccccccccccccccccccc} 
		\toprule	
		p\backslash k&4 &6&8&10&12&14&16&18&20&22&24&26&28&30&32&34
		\\
		\toprule
		2&0&0&0&0&0&0&1&1&1&2&3&3&6&5&8&9
		\\\toprule
		3&0&0&0&0&1&1&2&5&4&7&11&12&16&22&24&30
		\\\toprule
		5&0&0&0&1&2&4&7&11&15&22&29&38&49&61&75&92
		\\
		\toprule
		7&0&0&1&2&5&9&15&23&32&45&60&78&100&124&153&186
		\\
		\toprule
		11&0&0&1&5&11&21&35&54&78&109&146&191&244&306&377&459
		\\
		\toprule
		13&0&1&4&10&20&35&56&84&118&163&216&280&356&443&544&660\\
		\toprule
		17&0&1&6&15&32&57&92&139&198&273&364&473&602&751&924&1121
		\\
		\toprule
		19&0&2&8&21&42&74&118&177&252&346&460&597&758&946&1162&1409
		\\
		\bottomrule
		\end{array}\]
	\end{table}

\begin{table}[H]
	\caption{$s_k(p,{\rm IIIa+VIa/b})$ for $3\leq k\le30$ and $2\le p<20$.}
		\label{Table for mIIIa/VI odd}
		\renewcommand{\arraystretch}{0.8}
		\renewcommand{\arraycolsep}{.23cm}
		\[ \begin{array}{cccccccccccccccccccccccccc} 
		\toprule	
		p\backslash k&3&5&7&9&11&13&15&17&19&21&23&25&27&29
		\\
		\toprule
		2&0&0&0&0&0&0&0&0&0&0&0&1&1&2
		\\
		\toprule
		3&0&0&0&0&0&0&0&1&1&2&4&6&8&13
		\\
		\toprule
		5&0&0&0&0&1&2&6&9&17&24&37&50&70&89
		\\
		\toprule
		7&0&0&0&2&6&11&24&40&59&91&128&170&230&297
		\\
		\toprule
		11&0&0&3&13&33&68&121&195&295&424&585&785&1023&1306
		\\
		\toprule
		13&0&1&8&28&68&124&224&350&521&744&1026&1350&1770&2242
		\\
		\toprule
		17&0&2&23&70&165&312&537&837&1247&1756&2404&3178&4120&5211
		\\
		\toprule
		19&0&5&32&107&241&451&772&1207&1771&2511&3418&4509&5839&7391
		\\
		\bottomrule
		\end{array}\]
	\end{table}

	\begin{table}[H]
		\label{Table for mIIIa/VI even}
		\renewcommand{\arraystretch}{0.8}
		\renewcommand{\arraycolsep}{.21cm}
		\[ \begin{array}{cccccccccccccccccccccccccc} 
		\toprule	
		p\backslash k&4&6&8&10&12&14&16&18&20&22&24&26&28&30\\
		\toprule
		2&0&0&0&0&1&2&3&5&6&8&10&14&16&20
		\\
		\toprule
		3&0&0&1&2&4&7&11&15&20&27&33&43&53&63
		\\
		\toprule
		
		5&0&1&4&8&16&25&39&54&75&97&127&159&199&240
		\\
		\toprule
		7&0&3&9&21&38&60&93&133&178&241&311&391&491&602
		\\
		\toprule
		11&0&9&30&67&123&202&308&444&614&822&1071&1367&1710&2107
		\\
		\toprule
		13&3&18&52&111&203&323&500&715&987&1324&1732&2195&2766&3401
		\\
		\toprule
		17&5&35&106&226&417&681&1047&1510&2104&2821&3699&4725&5942&7330
		\\
		\toprule
		19&7&49&142&311&569&931&1434&2079&2881&3889&5092&6509&8193&10127
		\\
		\bottomrule
		\end{array}\]
	\end{table}
	
	\begin{table}[H]
	\caption{$s_k(p,{\rm IVa})$ for $3\leq k\le29$ and $2\le p<20$.}
		\label{Table for mIVa odd}
		\renewcommand{\arraystretch}{0.8}
		\renewcommand{\arraycolsep}{.08cm}
		\[ \begin{array}{ccccccccccccccccccccc} 
		\toprule	
		p\backslash k&3&5&7&9&11&13&15&17&19&21&23&25&27&29
		\\	\toprule
		2&0&0&0&0&0&1&1&2&4&5&6&11&11&15	\\	\toprule
		3&0&0&0&1&2&6&9&16&24&35&46&66&81&106	\\	\toprule	
		5&0&1&7&17&39&75&121&188&279&385&522&693&884&1116
		\\	\toprule
		7&0&8&31&88&181&332&541&832&1201&1678&2253&2962&3789&4774
		\\
		\toprule
		11&2&56&235&610&1255&2260&3661&5576&8055&11170&14995&19640&25101&31536
		\\
		\toprule
		13&8&118&477&1232&2529&4514&7335&11136&16065&22268&29891&39082&49985&62748
		\\
		\toprule
		17&23&362&1456&3728&7632&13606&22058&33456&48240&66800&89622&117146&149744&187920
		\\
		\toprule
		19&38&578&2295&5892&12033&21430&34739&52680&75897&105126&140995&184256&235521&295554
		\\
		\bottomrule
		\end{array}\]
	\end{table}

	\begin{table}[H]
		\label{Table for mIVa even}
		\renewcommand{\arraystretch}{0.8}
		\renewcommand{\arraycolsep}{.125cm}
		\[ \begin{array}{ccccccccccccccccccccc} 
		\toprule	
		p\backslash k&4&6&8&10&12&14&16&18&20&22&24&26&28
		\\	\toprule
		2&0&0&0&1&1&2&2&3&6&7&9&13&14
		\\	\toprule
		3&0&1&2&6&9&16&22&33&46&62&79&104&126
		\\	\toprule		5&1&7&20&41&76&128&193&282&398&532&700&904&1132
		\\	\toprule
		7&5&26&73&162&297&498&767&1126&1575&2138&2811&3624&4567
		\\	\toprule
		11&25&150&445&984&1839&3100&4805&7070&9945&13504&17819&23000&29045
		\\	\toprule
		13&51&292&869&1930&3619&6084&9471&13926&19597&26628&35167&45360&57353
		\\	\toprule
		17&144&848&2550&5674&10672&17984&28016&41238&58090&78960&104336&134656&170294
		\\	\toprule
		19&225&1326&3979&8888&16713&28170&43911&64660&91061&123846&163647&211212&267157
		\\	\bottomrule
		\end{array}\]
	\end{table}

\begin{table}[H]
	\caption{$s_k(p,{\rm Va})$  for $3\leq k\le30$ and $2\le p<20$.}
		\label{Table for mVa odd}
		\renewcommand{\arraystretch}{0.8}
		\renewcommand{\arraycolsep}{.23cm}
		\[ \begin{array}{cccccccccccccccccccccccccc} 
		\toprule	
		p\backslash k&3&5&7&9&11&13&15&17&19&21&23&25&27&29\\
		\toprule
		2&0&0&0&0&0&0&1&1&1&2&2&3&4&5\\
		\toprule
		3&0&0&0&0&1&1&3&4&6&7&11&13&17&21\\
		\toprule
		5&0&0&1&4&7&11&19&27&36&51&66&82&106&130\\
		\toprule
		7&0&1&5&10&21&33&54&76&109&144&192&243&309&378\\
		\toprule
		11&0&3&15&38&76&125&205&298&420&573&761&970&1240&1533\\
		\toprule
		13&0&8&28&66&127&216&339&500&705&959&1267&1635&2067&2569\\
		\toprule
		17&0&15&56&141&274&467&746&1106&1560&2141&2837&3661&4653&5794\\
		\toprule
		19&1&20&81&192&381&652&1037&1536&2187&2982&3967&5126&6513&8106\\
		\bottomrule
		\end{array}\]
	\end{table}
	
	\begin{table}[H]
		\label{Table for mVa even}
		\renewcommand{\arraystretch}{0.8}
		\renewcommand{\arraycolsep}{.23cm}
		\[ \begin{array}{cccccccccccccccccccccccccc} 
		\toprule	
	p\backslash k&4&6&8&10&12&14&16&18&20&22&24&26&28&30\\
		\toprule
		2&0&0&0&0&0&0&0&0&0&0&0&0&0&1\\
		\toprule
		3&0&0&0&0&0&0&1&0&2&2&3&4&7&7\\
		\toprule
		5&0&0&0&1&2&3&7&11&15&24&33&42&58&74\\
		\toprule
		7&0&0&1&3&8&14&26&39&61&84&118&154&203&255\\
		\toprule
		11&0&0&5&18&43&75&135&205&300&423&578&750&980&1230\\
		\toprule
		13&0&3&14&39&82&149&245&375&545&759&1023&1342&1721&2166\\
		\toprule
		17&0&8&35&101&208&368&608&923&1325&1848&2480&3232&4147&5205\\
		\toprule
		19&0&12&58&146&306&540&882&1330&1922&2652&3564&4644&5944&7442\\
		\bottomrule
		\end{array}\]
	\end{table}

\end{appendix}

\vspace*{0.1in}
\textbf{Acknowledgments.} We would like to thank Tomoyoshi Ibukiyama, Kimball Martin, Cris Poor, Sug Woo Shin and Satoshi Wakatsuki for their helpful comments. We would also like to thank the referee for the detailed comments and suggestions. 

\bibliographystyle{plain}
\bibliography{On_counting_cuspidal_automorphic_representations_for_GSp4.bib}

\vspace{5ex}
\noindent Department of Mathematics, Fordham University, Bronx, New York 10458, USA.

\noindent E-mail address: {\tt manami.roy.90@gmail.com}

\vspace{2ex}
\noindent Department of Mathematics, University of North Texas, Denton, TX 76203-5017, USA.

\noindent E-mail address: {\tt ralf.schmidt@unt.edu}

\vspace{2ex}
\noindent Department of Mathematics, University of South Carolina, Columbia, SC 29208, USA.

\noindent E-mail address: {\tt yishaoyun926@gmail.com}
\end{document}